\documentclass[11pt,a4paper]{article}
\usepackage{geometry}

\usepackage{amsmath,amsthm,amsfonts,amssymb,array,graphicx,epsfig,fancyhdr,stmaryrd,algpseudocode}
\usepackage{enumerate}
\usepackage{cite}

\usepackage[english]{babel}
\usepackage[T1]{fontenc}
\usepackage[utf8]{inputenc}

\usepackage{qtree}
\usepackage{color}

\newtheorem{theorem}{Theorem}[section]

\newtheorem{proposition}[theorem]{Proposition}
\newtheorem{lemma}[theorem]{Lemma}
\theoremstyle{definition}

\newtheorem{remark}[theorem]{Remark}

\newcommand{\id}{\mathrm{Id}}

\newcommand{\Acal}{\ensuremath{{\mathcal{A}}}}

\newcommand{\N}{\ensuremath{{\mathbb N}}}

\newcommand{\R}{\ensuremath{{\mathbb R}}}

\newcommand{\be}{\begin{enumerate}}
\newcommand{\ee}{\end{enumerate}}

\newcommand{\Hcal}{\mathcal{H}}

\DeclareMathOperator{\sgn}{sgn}

\newcommand*{\argmin}{\operatornamewithlimits{argmin}\limits}

\DeclareMathOperator{\rank}{rank}
\DeclareMathOperator{\prox}{prox}
\DeclareMathOperator{\diag}{diag}

\def\Chi{\raise .3ex\hbox{\large $\chi$}}

\newcommand{\bu}{{\bf u}}
\newcommand{\bw}{{\bf w}}
\newcommand{\bv}{{\bf v}}

\newcommand{\weakl}[1]{\ensuremath{\ell_{#1,\infty}}}

\providecommand{\abs}[1]{\lvert#1\rvert}

\providecommand{\norm}[1]{\lVert#1\rVert}
\providecommand{\bignorm}[1]{\bigl\lVert#1\bigr\rVert}

\newcommand{\bS}{\mathbf{S}}
\newcommand{\br}{\mathbf{r}}
\newcommand{\cF}{\mathcal{F}}
\newcommand{\cG}{\mathcal{G}}
\newcommand{\cH}{\mathcal{H}}
\newcommand{\cS}{\mathcal{S}}

\newcommand{\HS}{{\mathrm{HS}}}

\usepackage{algorithm}
\usepackage{algorithmicx}
\usepackage{algpseudocode}
\algrenewcommand\algorithmicrequire{\textrm{parameters:}}
\algrenewcommand\algorithmicensure{\textrm{output:}}
\algrenewcommand\algorithmicfunction{\textrm{function}}
\algrenewcommand\algorithmicwhile{\textrm{while}}
\algrenewcommand\algorithmicdo{}
\algrenewcommand\algorithmicend{\textrm{end}}
\algrenewcommand\algorithmicforall{\textrm{for all}}
\algrenewcommand\algorithmicfor{\textrm{for}}
\algrenewcommand\algorithmicrepeat{\textrm{repeat}}
\algrenewcommand\algorithmicuntil{\textrm{until}}
\algrenewcommand\algorithmicif{\textrm{if}}
\algrenewcommand\algorithmicelse{\textrm{else}}
\algrenewcommand\algorithmicthen{\textrm{then}}

\DeclareMathOperator{\solve}{\textsc{STSolve}}
\DeclareMathOperator{\iesolve}{\textsc{InexactSTSolve}}

\title{Iterative Methods Based on Soft Thresholding of Hierarchical Tensors\thanks{This research was supported in part by DFG SPP 1324 and ERC AdG BREAD.}}
\date{\today}
\author{Markus Bachmayr and Reinhold Schneider}

\begin{document}
\maketitle

\begin{abstract}
We construct a soft thresholding operation for rank reduction of hierarchical tensors and subsequently consider its use in iterative thresholding methods, in particular for the solution of discretized high-dimensional elliptic problems. 
The proposed method for the latter case automatically adjusts the thresholding parameters, based only on bounds on the spectrum of the operator, such that the arising tensor ranks of the resulting iterates remain quasi-optimal with respect to the algebraic or exponential-type decay of the hierarchical singular values of the true solution. In addition, we give a modified algorithm using inexactly evaluated residuals that retains these features. The practical efficiency of the scheme is demonstrated in numerical experiments.
\end{abstract}

\section{Introduction}

Subspace-based tensor representations such as the \emph{hierarchical Tucker format} \cite{Hackbusch:09} or the special case of the \emph{tensor train format} \cite{Oseledets:09} enable the numerical treatment of problems in very high dimensions by exploiting particular low-rank structures. Here one has a well-defined notion of tensor rank as a tuple of matrix ranks of certain matricizations. From a computational perspective, these tensor formats have the major advantage that for any tensor given in such a representation, by a simple combination of linear algebraic procedures, one may obtain an error-controlled, quasi-optimal approximation by a tensor of lower ranks; this can be achieved by truncating the ranks of a \emph{hierarchical singular value decomposition} \cite{vidal,Grasedyck:10,Oseledets:11}, or HSVD for short, of the tensor.

In this work, we consider an alternative procedure for reducing ranks that is based on \emph{soft thresholding} of the singular values in a HSVD, as opposed to the mentioned rank truncation (which would correspond to their hard thresholding). The new procedure has similar complexity and quasi-optimality properties, but unlike the truncation it is \emph{non-expansive}, which turns out to be a major advantage in the context of iterative schemes.

A large part of the results that we obtain are in fact applicable to fixed point iterations based on general contractive mappings.
The iterative scheme that we focus on as an example, however, is a Richardson iteration for solving the linear equation  
\begin{equation}\label{opeq}
\mathcal{A} \mathbf{u} = \mathbf{f} 
\end{equation}
on a separable tensor product Hilbert space 
\[  \mathcal{H}   = \mathcal{H}_1 \otimes \cdots\otimes\mathcal{H}_d  \,, \]
where $\mathbf{f} \in \mathcal{H} $ is given and
$\mathcal{A} : \mathcal{H} \to \mathcal{H} $
is elliptic on $\mathcal{H}$; that is, the iteration is of the form
\begin{equation}\label{truncatedrichardson}
  \bu_{k+1} = \mathbf{S}_{\alpha_k} \bigl( \bu_k - \mu (\mathcal{A} \bu_k - \mathbf{f}) \bigr) \,,
\end{equation}
where $\mathbf{S}_{\alpha}$ is the proposed soft thresholding procedure with suitable parameters $\alpha_k$.

Even when both $\mathcal{A}$ and $\mathbf{f}$ have exact low-rank representations, the unique solution $ \mathbf{u}^* $ of the problem \eqref{opeq} may no longer be of low rank. It turns out, however, that in many cases of interest, $\mathbf{u^*}$ can still be efficiently \emph{approximated} by low-rank tensors up to any given error tolerance. Here one can obtain algebraic error decay with respect to the ranks under fairly general conditions \cite{SchneiderUschmajew:14,KressnerUschmajew:14}, and superalgebraic or exponential-type decay in more specific situations \cite{KressnerTobler:11,Grasedyck:04,DDGS:14}. 

When a solution $\bu^*$ that has this property is approximated by an iteration such as \eqref{truncatedrichardson}, it is not a priori clear to what extent also the iterates $\bu_{k}$ retain comparably low ranks, since the basic iteration without truncation can in principle lead to an exponential rank increase.
That the ranks of $\bu_k$ remain comparable to those needed for approximating $\bu^*$ at the current accuracy therefore depends crucially on the appropriate choice of thresholding parameters $\alpha_k$. Keeping the tensor ranks of iterates as low as possible is of crucial importance for the efficiency of such methods, since the complexity of the tensor operations that need to be performed grows like the fourth power of these ranks.

We show in this work that when the rank reduction in each step is done by the tensor soft thresholding procedure that we shall describe, \emph{quasi-optimal} tensor ranks can be enforced \emph{for each single iterate} $\bu_k$, irrespective of the rank increase caused by $\Acal$. 
This means that, assuming that the hierarchical singular values of $\bu^*$ have either algebraic or exponential-type decay, the tensor ranks of each $\bu_k$ can be bounded up to a uniform constant by the maximum rank of the best hierarchical tensor approximation to $\bu^*$ of the same accuracy.
To this end, we exploit the non-expansiveness of soft thresholding, which allows us to choose the thresholding parameters in each step as large as required to  control the ranks, without compromising the convergence of the iteration. 

After describing the construction of the operation $\bS_{\alpha_k}$, we begin by identifiying choices of geometrically decreasing $\alpha_k$ that lead to the desired rank estimates provided that the precise decay of the hierarchical singular values of $\bu^*$ is known. On this basis, we then construct a scheme which, based on a certain a posteriori criterion, adjusts $\alpha_k$ to the unknown decay of hierarchical singular values such that the quasi-optimal ranks are preserved.
This method requires no a priori information except for bounds on the spectrum of $\Acal$ and the norm of $\mathbf{f}$.
In a third step, we develop a perturbed version of the scheme that permits inexactly evaluated residuals.

For the case that the rank reduction is done by a truncated HSVD (that is, by hard thresholding), a scheme for choosing thresholding parameters that leads to near-optimal ranks is given in \cite{BD1,BD2}; to the authors' knowledge, this is the only previous instance of a method that guarantees global converge to the true solution while at the same time, the arising ranks can be estimated in terms of the ranks required for the approximation of the solution.
A limitation of the approach used there to control the ranks is that their near-optimality is enforced by truncating with a sufficiently high tolerance, which can be done only after every few iterations when a certain error reduction has been achieved. The ranks of intermediate iterates can therefore still accumulate in the iterations between these complexity reductions, which can be problematic if each application of the operator already causes a large rank increase.
In the method proposed here, this accumulation can be ruled out, and intermittent, sufficiently large increases of approximation errors that restore quasi-optimality are no longer necessary.

Note that here we do not address the aspect of an adaptive underlying discretization of the problem as considered in \cite{BD1,BD2}. Although our resulting procedure of the form \eqref{truncatedrichardson} can in principle be formulated on infinite-dimensional Hilbert spaces, in this work we restrict our considerations concerning a numerically realizable version to \emph{fixed discretizations}.
In other words, in the form given here, the scheme applies either to infinite-dimensional $\mathcal{H}_i \simeq \ell_2(\N)$ (which is of course not implementable in practice), or to a fixed finite-dimensional choice $\mathcal{H}_i \simeq \R^{n_i}$.
The version of the algorithm allowing inexact evaluation of residuals can, however, serve as a starting point for combining the method with adaptive concepts for identifying suitable subsets of $\ell_2(\N)$ in the course of the iteration. 
Furthermore, we expect that the concepts put forward here can also be used in the construction of adaptive methods for sparse basis representations.

Iterations using soft thresholding of sequences have been studied extensively in the context of inverse and ill-posed problems, see e.g.\ \cite{Daubechiesetal:08,BrediesLorenz:08,BeckTeboulle}, where they are especially well suited for obtaining convergence under very general conditions.
Note that in such a setting, a priori choices of geometrically decreasing thresholding parameters have been proposed, e.g., in \cite{Wrightetal:09} and \cite{DahlkeFornasierRaasch:12}.
However, our approach for controlling the complexity of iterates -- in the present case, the arising tensor ranks -- in iterative schemes for well-posed problems appears to be new, in particular the a posteriori criterion that steers the decrease of the thresholding parameters.

The proposed method can also be motivated by a variational formulation of the problem. For instance, if $\Acal$ is symmetric, the solution is characterized by
\[
 \mathbf{u}^* =  \argmin_{\mathbf{v} \in \mathcal{H}}  \Bigl\{  \frac{1}{2} \langle  \mathcal{A} \mathbf{v}  , \bv \rangle 
       - \langle  \mathbf{f}, \bv \rangle  \Bigr\}  \,.
\]
A standard way to solve this problem in the spirit of 
a Ritz-Galerkin method would be to restrict it to the manifold of hierarchical tensors with 
fixed rank and treat the resulting minimization problem by Riemannian gradient methods 
or alternating least squares techniques \cite{HRS}.
However, the sets over which one needs to minimize are not convex, and there generally exist many local minima.
Roughly speaking, in such methods one fixes the model class (in this case by the admissible hierarchical tensor ranks) and
attempts to  minimize the error over this class.

In an alternative variational formulation, one can prescribe an error tolerance, for instance
$  \| \mathcal{A}\bv - \mathbf{f} \| \leq \varepsilon$, and attempt to minimize the tensor ranks over the set of such $\bv$. Although the admissible set is then convex, even in the matrix case $d=2$ the rank does not give a convex functional. However, one can instead minimize an appropriate convex relaxation such as the $\ell_1$-norm of singular values. It is well known that in the matrix case, such relaxed problems can be solved by proximal gradient methods, which can be rewritten as iterative soft thresholding \cite{Ma:11} and hence take precisely the form \eqref{truncatedrichardson} when $d=2$. In this case, our method can therefore also be motivated as a rank minimization scheme, although this connection does not play a role in the analysis. Note, however, that in the case of higher-order tensors, where our soft thresholding procedure no longer permits an interpretation of the resulting scheme as a proximal gradient method, this is only a formal analogy.

This article is arranged as follows: in Section \ref{sec:prelim}, we collect some prerequisites concerning the hierarchical tensor format as well as soft thresholding of sequences and of Hilbert-Schmidt operators. In Section \ref{sec:hierarchst}, we then describe and analyze the new soft thresholding procedure for hierarchical tensors.
In Section \ref{sec:fpst}, we consider the combination of this procedure with general contractive fixed point iterations and derive rank estimates for sequences of thresholding parameters that are chosen based on a priori information on the tensor approximability of $\bu^*$. In Section \ref{sec:aposteriori}, we introduce an algorithm that automatically determines a suitable choice of thresholding parameters without using information on $\bu^*$, analyze its convergence, and additionally give a modified version of the scheme based on inexact residual evaluations. In Section \ref{sec:numexp}, we conclude with numerical experiments on a discretized high-dimensional Poisson problem that illustrate the practical performance of the proposed method.

Throughout this paper, the notation $A\lesssim B$ is used to indicate that there exists a constant $C>0$ such that $A \leq C B$.

\section{Preliminaries}\label{sec:prelim}

By $\norm{\cdot}$, we always denote either the canonical norm on $\Hcal$, which is the product of the norms on the $\Hcal_i$, or the $\ell_2$-norm when applied to a sequence.

To quantify the sparsity of sequences, we shall frequently use membership in \emph{weak-$\ell_p$ spaces}, which are defined as follows:
For a given sequence $a = (a_k)_{k\in\N}$, for each $n\in \N$ let $a^*_n$ be the $n$-th largest of the values $\abs{a_k}$. Then for $p>0$, the space $\weakl{p}$ is defined as the collection of sequences for which
\[    \abs{a}_{\weakl{p}}  :=  \sup_{n\in\N} n^{\frac1p}  a^*_n \]
is finite, and this quantity defines a (quasi-)norm on $\weakl{p}$. We will use these spaces with $p<2$, which implies $\weakl{p} \subset \ell_2$; note that one always has $\ell_p \subset \weakl{p} \subset \ell_{p'}$ for all $p' > p$.

For separable Hilbert spaces $\cS_1$, $\cS_2$, we write $\HS(\cS_1, \cS_2)$ for the space of Hilbert-Schmidt operators from $\cS_1$ to $\cS_2$ with the 
Hilbert-Schmidt norm $\norm{\cdot}_\HS$, which reduces to the Frobenius norm in the case of finite-dimensional spaces.
Hilbert-Schmidt operators have a singular value decomposition with singular values in $\ell_2$,  satisfying the following perturbation estimate, cf.\ \cite{Mirsky:60}.

\begin{theorem}\label{thm:mirsky}
Let $\cS_1$, $\cS_2$ be separable Hilbert spaces, let $\mathbf{X}, \mathbf{\tilde X} \in \HS(\cS_1, \cS_2)$, and let $\sigma, \tilde\sigma \in \ell_2(\N)$ denote the corresponding sequences of singular values. Then
$\norm{ \sigma - \tilde\sigma}_{\ell_2(\N)}  \leq \norm{ \mathbf{X} - \mathbf{\tilde X} }_{\HS}$.
\end{theorem}

Note that this was shown for matrices in \cite{Mirsky:60}, but the proof immediately carries over to Hilbert-Schmidt operators.

\subsection{The Hierarchical Tensor Format}

We now briefly recall definitions and facts concerning the hierarchical Tucker format \cite{Hackbusch:09} and collect some basic observations that will play a role later.
For further details on the hierarchical format, we refer to \cite{Hackbusch:12}.

Throughout this work, we assume $d \geq 2$. Let $\mathbb{T}$ be a binary dimension tree for tensor order $d$, that is, with \emph{root element} $\{1,\ldots,d\}$; examples for $d=4$ are given in Figure \ref{htequiv}. We adopt the terminology of \cite{Grasedyck:10}, referring to the collections of basis vectors in the leaves of the tree as \emph{mode frames}, and to the coefficient tensors at interior nodes of the tree as \emph{transfer tensors}.

We shall later make use of a certain equivalence between dimension trees, which we formulate in terms of the \emph{edges} in these trees.
For each node $n\in \mathbb{T}$, we set $[n]:= \{ 1,\ldots, d\}\setminus n$. In general, $[n]\notin \mathbb{T}$.
Let
$$   \mathbb{E} := \bigl\{  \{ n , [ n]\} \colon n\in\mathbb{T}\setminus \{1,\ldots,d\}  \bigr\} \,. $$
Then the elements of $\mathbb{E}$ correspond precisely to the edges in the tree $\mathbb{T}$, where the root element $ \{1,\ldots,d\}$ is regarded as part of an edge. We set $E:= \#\mathbb{E} = 2 d - 3$.

For a given set of edges, there are several dimension trees that correspond to the same matricizations of the tensor, but have the root element of the tree at a different edge. This is illustrated in Figure \ref{htequiv} for a tensor of order four. Moving the root element in the tensor representation can be done in practice by basic linear algebra manipulations, where the existing component tensors in the representation are simply relabelled and reorthogonalized accordingly. For instance, in passing from the first to the second tree in Figure \ref{htequiv}, the transfer tensor for node $\{2,3,4\}$ is relabelled to $\{1,3,4\}$.

For what follows, we always assume a fixed enumeration $\{ n_t, [n_t] \}$, $t=1,\ldots, E$, of $\mathbb{E}$. Note that the efficiency of the algorithms we will describe may depend on this sequence. For practical purposes, it should be chosen such that moving the root element from one edge to the next in the enumeration takes as little work as possible, as for instance in Figure \ref{htequiv}.
For each $t$, we denote by $\mathcal{M}_t(\bu)$ the \emph{matricization} corresponding to the $t$-th edge of the tensor $\bu$, which for infinite-dimensional $\Hcal$ is a Hilbert-Schmidt operator
\begin{equation*}
   \mathcal{M}_t(\bu) \colon \bigotimes_{i \in [n_t]} \Hcal_i  \to \bigotimes_{i \in n_t} \Hcal_i \,,
\end{equation*}
and by $\mathcal{M}^{-1}_t$ we denote the mapping that converts a matricization back to a tensor.
Note that for each $t$, one has $\mathcal{M}_t(\bu) + \mathcal{M}_t(\bv) = \mathcal{M}_t(\bu+\bv)$ and 
\begin{equation}\label{matricizationnorm} 
   \norm{\bu} = \norm{\mathcal{M}_t(\bu)}_\HS \,.
\end{equation}
The sequence of singular values of this matricization is denoted by $\sigma_t (\bu)$, which we always assume to be defined on $\N$ (with extension by zero in the case of finite-dimensional $\cH_i$), and we set
\[   \rank_t(\bu) := \rank  \bigl(\mathcal{M}_t(\bu)\bigr) = \# \{  k\in\N \colon \sigma_{t,k}(\bu) \neq 0 \} \,.   \]
\begin{figure}
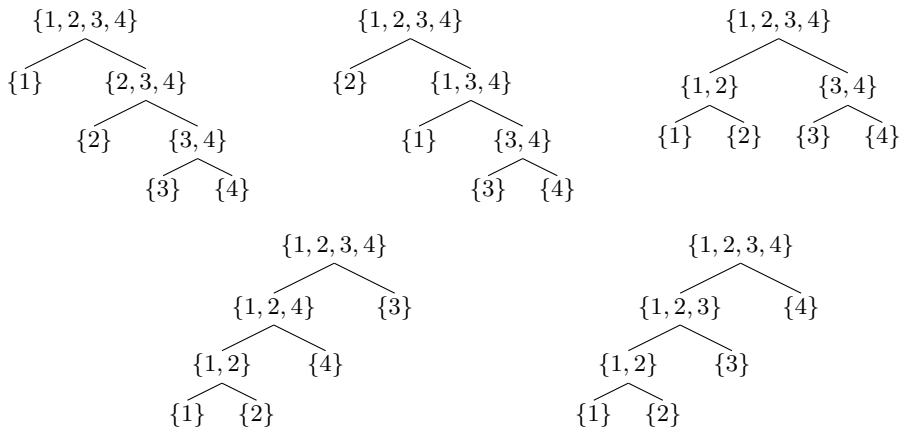
\centering
{\footnotesize\Tree[.$\{1,2,3,4\}$ [.$\{1\}$ ]
        [.$\{2,3,4\}$ 
              [.$\{2\}$ ] [.$\{3,4\}$  [.$\{3\}$ ] [.$\{4\}$ ] ] ] ] }  
{\footnotesize\Tree[.$\{1,2,3,4\}$ [.$\{2\}$ ]  [.$\{1,3,4\}$
       [.$\{1\}$ ] [.$\{3,4\}$ 
              [.$\{3\}$ ] [.$\{4\}$ ] ] ] ]}
{\footnotesize\Tree[.$\{1,2,3,4\}$ [.$\{1,2\}$ 
       [.$\{1\}$ ] [.$\{2\}$ ] ] 
       [.$\{3,4\}$ 
              [.$\{3\}$ ] [.$\{4\}$ ] ] ] } 
\vspace{9pt}
              
{\footnotesize\Tree[.$\{1,2,3,4\}$ 
        [.$\{1,2,4\}$ 
               [.$\{1,2\}$  [.$\{1\}$ ] [.$\{2\}$ ] ] [.$\{4\}$ ] ] [.$\{3\}$ ] ] }             
{\footnotesize\Tree[.$\{1,2,3,4\}$ 
        [.$\{1,2,3\}$ 
               [.$\{1,2\}$  [.$\{1\}$ ] [.$\{2\}$ ] ] [.$\{3\}$ ] ] [.$\{4\}$ ] ] }    
\caption{Examples of equivalent dimension trees obtained by moving the root element.}\label{htequiv}
\end{figure}

\subsection{Soft  Thresholding}

For $x \in \R$, soft thresholding with parameter $\alpha > 0$
 is defined by 
\[  s_{\alpha } (x) :=  \sgn(x) \max \{ \abs{x} - \alpha , 0 \} \,. \]
In comparison, hard thresholding is given by 
$  h_{\alpha} (x)  := \bigl(1 - \Chi_{[-\alpha, \alpha]}(x) \bigr) x$.

Applied to each element of a vector or sequence, hard thresholding provides a very natural means of obtaining sparse approximations by dropping entries of small absolute value, which is closely related to best $n$-term approximation \cite{Devore:98}.
In contrast, soft thresholding not only replaces entries that have absolute value below the threshold by zero, but also decreases all remaining entries, incurring an additional error. However, this operation has a non-expansiveness property that is useful in the construction of iterative schemes, and that can be derived from a variational characterization.

To describe this characterization, for a proper, closed convex functional $\mathcal{J} : \mathcal{G}  \to \mathbb{R} $ on a Hilbert space $\mathcal{G}$ and $ \alpha \geq 0$, following \cite{Moreau:65} we define the {\em proximity operator}  $\prox^{\alpha}_{\mathcal{J}} : \mathcal{G} \to \mathcal{G} $ by 
\begin{equation} \label{eq:prox}
\prox^{\alpha}_{\mathcal{J}} (\mathbf{u})  := \argmin_{\mathbf{v} \in \mathcal{G} }   
 \Bigl\{  \alpha \mathcal{J}(\mathbf{v}) + \frac{1}{2 } \norm{ \mathbf{u} - \mathbf{v} }^2_\mathcal{G}   \Bigr\}  \ . 
\end{equation}
As shown in \cite{Moreau:65}, such operators have the following general property, which plays a crucial role in this work.

\begin{lemma}\label{nonexp}
The proximity  operator $ \prox^{\alpha }_{\mathcal{J}} $ is {\em non-expansive}, that is,
$$
\norm{ \prox^{\alpha }_{\mathcal{J}} (\mathbf{u} )   - \prox^{\alpha }_{\mathcal{J}} (\mathbf{v} ) }_\mathcal{G}
   \leq  \norm{ \mathbf{u}    -  \mathbf{v}  }_\mathcal{G} \,,\quad \bu,\bv\in\mathcal{G}\,.
$$
\end{lemma}

Soft thresholding of a sequence $\bv$ by applying $s_\alpha$ to each entry, $s_\alpha(\bv) := \bigl(s_\alpha(v_i)\bigr)_{i\in\N}$, can be characterized as the proximity operator of the functional $\norm{\cdot}_{\ell_1}$, see e.g.\ \cite{Daubechiesetal:08}, and is therefore a non-expansive mapping on $\ell_2$.

An analogous characterization is still possible when soft thresholding is applied to the singular values of matrices or operators, which provides a reduction to lower matrix ranks.
More precisely, the soft thresholding operation $S_\alpha$ for matrices is defined as follows: for a given matrix $\mathbf{X}$ with singular value decomposition 
$$ \mathbf{X} = \mathbf{U} \diag \bigl( {\sigma}_i (\mathbf{X})\bigr) \mathbf{V}^T  \,,  $$
where ${\sigma}_i (\mathbf{X})$ denotes the $i$-th singular value of $\mathbf{X}$,
we set
$$ S_{\alpha} (\mathbf{X})  : =  \mathbf{U} \diag \bigl(s_{\alpha} ({\sigma}_i(\mathbf{X}) ) \bigr)  \mathbf{V}^T \,.
$$
Note that application of the hard thresholding $h_\alpha$ to the singular values would instead correspond to a rank truncation of the singular value decomposition.
For Hilbert-Schmidt operators $\mathbf{X}$ one can define $S_\alpha(\mathbf{X})$ analogously.

The mapping $S_\alpha$ is the proximity operator for the \emph{nuclear norm},
$$   \norm{\mathbf{X}}_* :=   \norm{\sigma(\mathbf{X})}_{\ell_1} =  \sum_{i \geq 1} \sigma_i(\mathbf{X})\,.   $$

\begin{lemma} \label{lem:soft}
Let $\cS_1, \cS_2$ be separable Hilbert spaces, $\mathbf{X} \in \HS(\cS_1, \cS_2)$,  and $\alpha \geq 0$. Then
$$ \prox^\alpha_{\norm{\cdot}_*} (\mathbf{X} ) =  S_{\alpha }  (\mathbf{X})   \,, $$
or in other words, 
  $$ S_{\alpha}  (\mathbf{X})  
  =  \argmin_{\mathbf{V} \in \HS} \Bigl\{ \alpha \| \mathbf{V}\|_{*}  +  \frac{1}{2} \|  \mathbf{X} - \mathbf{V} \|^2_\HS \Bigr\}   \,.$$
\end{lemma}

Note that this statement is shown for finite matrices $\mathbf{X}$, e.g., in \cite{CaiCandesShen:08}; the generalization to Hilbert-Schmidt operators can be obtained by finite-dimensional approximation, using Theorem \ref{thm:mirsky} and that the $\ell_1$-norm is lower semicontinuous with respect to componentwise convergence of sequences, which follows from Fatou's lemma.

%\begin{frame}
\section{Soft  Thresholding of Hierarchical Tensors}\label{sec:hierarchst}

In this section, we construct a non-expansive soft thresholding operation for the rank reduction of hierarchical tensors.
By $S_{t,{\alpha} }: \mathcal{H} \to \mathcal{H}$  we denote soft thresholding applied to the matricization $\mathcal{M}_t(\cdot)$ of the input,
\[   S_{t,\alpha} (\bu) := \mathcal{M}^{-1}_t \circ S_{\alpha} \circ \mathcal{M}_t(\bu)  \,.  \]
The complete soft shrinkage operator $ \mathbf{S}_{{\alpha}} : \Hcal \to \Hcal $ is then given as the successive application of this operation to each matricization, that is,
\begin{equation}
\label{Salphadef}
  \mathbf{S}_{{\alpha}  } (\mathbf{u} )  : = S_{E, \alpha  } \circ \ldots \circ S_{1, \alpha } ( \mathbf{u} )   \ . 
\end{equation}
Clearly, the result of applying $\mathbf{S}_\alpha$ for $\alpha > 0$ always has finite (but not a priori fixed) hierarchical ranks. 
 
For a hierarchical tensor $\mathbf{u}$ with suitably numbered edges $t = 1, \ldots, E$, 
the soft thresholding $ \mathbf{S}_{{\alpha}} (\mathbf{u}) $ can be obtained as follows: starting with $\bu_0 = \bu$, for each $t$, we first rearrange $\bu_{t-1}$ such that the root element is on edge $t$ with a singular value decomposition of $\mathcal{M}_t(\bu_{t-1})$, which exposes the singular values $\sigma_t(\bu_{t-1})$ and thus allows the direct application of $S_{t,\alpha}$ to obtain $\bu_t := S_{t,\alpha}(\bu_{t-1})$, and finally $\bS_\alpha(\bu) = \bu_E$. An example of an order in which the edges in $\mathbb{E}$ can be visited for this procedure in the case $d=4$ is given in Figure \ref{htequiv}.

\begin{remark}
Assuming that we are given a hierarchical tensor with $\dim \mathcal{H}_i \leq n \in \N$ and representation ranks bounded by $r$, then the first step of bringing this tensor into an initial HSVD representation takes $\mathcal{O}(d r^4 + dr^2n)$ operations. Moving the root element from one edge to an adjacent one costs $\mathcal{O}(r^4 + r^2 n)$ operations; if an appropriate ordering of the edges is used, the total complexity for applying $\bS_\alpha$ can thus be seen to be bounded by $\mathcal{O}(d r^4 + dr^2n)$ as well.
\end{remark}
  
\begin{proposition}\label{Salphanonexp}
For any $\bu, \bv\in\Hcal$ and $\alpha>0$, the operator $\mathbf{S}_\alpha$ defined in \eqref{Salphadef} satisfies $\norm{\mathbf{S}_\alpha(\bu) - \mathbf{S}_\alpha(\bv)} \leq \norm{\bu - \bv}$.
\end{proposition}

\begin{proof}
The statement follows by repeated application of Lemmata \ref{nonexp}, \ref{lem:soft} and \eqref{matricizationnorm}.
\end{proof}

The following lemma guarantees that applying soft thresholding to a certain matricization of a tensor does not increase the hierarchical singular values of any other matricization of this tensor.

\begin{lemma}\label{monotonicity}
For any $\bv\in\Hcal$ and for $t, s = 1, \ldots, E$, one has $\sigma_{t,i}( \bu) \geq \sigma_{t,i}( S_{s,\alpha} (\bu))$ for all $i\in\N$ and any $\alpha \geq 0$.
\end{lemma}

\begin{proof}
Note that for the action of $S_{s,\alpha}$, the tensor is rearranged such that the edge $s$ holds the root element. Thus the statement follows exactly as in part 3 of the proof of Theorem 11.61 in \cite{Hackbusch:12}, see also the proof of Theorem 7.18 in \cite{Kuehn:12}; there it is shown that when singular values are decreased at the root element, this cannot cause any singular value of any other matricization to increase.
\end{proof}

Using the above lemma, the error incurred by application of $\bS_\alpha$ to a tensor $\bu$ can be estimated in terms of the sequences of hierarchical singular values $\sigma_t(\bu)$, $t=1,\ldots, E$.

\begin{lemma}\label{lmm:st}
For any $\bu \in \Hcal$ and $r\in\N_0$, let
\[  \tau_{t,r}(\bu) := \inf_{\rank_t(\bw) \leq r} \norm{\bu - \bw} =  \Bigl(  \sum_{ k > r } \abs{\sigma_{t,k}(\bu)}^2  \Bigr)^{\frac12} \,.   \]
Furthermore, for any $\delta>0$ we define
\[  r_{t,\delta}(\bu) :=   \max  \bigl\{  r\in\N \colon \sigma_{t,r} > \delta \bigr\}  \cup \{ 0 \}   \,. \]
Then for any given $\alpha>0$,
\begin{equation}\label{Salphaapprox}
   \max_{t=1,\ldots,E} d^\alpha_t(\bu)  \leq \norm{ \mathbf{S}_\alpha(\bu) -  \bu}  \leq   \sum_{t=1}^E d^\alpha_t(\bu)  \,,    
\end{equation}
where, with $\tau_{t,\alpha}(\bu):=\tau_{t,r_{t,\alpha}(\bu)}(\bu) $,
\[   d_t^\alpha (\bu) := \bignorm{ \sigma_t \bigl(S_{t,\alpha}(\bu) \bigr) - \sigma_t(\bu) } = \sqrt{ \alpha^2 r_{t,\alpha}(\bu) +  \bigl(\tau_{t,\alpha}(\bu)\bigr)^2 }   \,. \]
\end{lemma}

\begin{remark}
It can be seen that the upper estimate in \eqref{Salphaapprox} is generally sharp by choosing $\bu$ as a tensor of rank one (that is, with all hierarchical ranks equal to one) with $\norm{\bu} = E\alpha$. In this case, $\mathbf{S}_\alpha(\bu) = 0$.
\end{remark}

\begin{proof}[Proof of Lemma \ref{lmm:st}]
We first show the second inequality in \eqref{Salphaapprox}. Let
\[   \bv_1 := \bu,\qquad  \bv_t := S_{t-1,\alpha} \circ \ldots\circ S_{1, \alpha}(\bu) , \quad t \geq 2 \,.  \]
By a telescoping sum argument, applying the soft thresholding error estimate to each individual application of $S_{\alpha,t}$, we obtain
$$ 
   \norm{ \mathbf{S}_\alpha(\bu) -  \bu}   \leq  \sum_{t=1}^E  \norm{S_{t,\alpha}(\bv_t) - \bv_t}  \leq   \sum_{t=1}^E d^\alpha_t (\bv_t )  \,.
$$
It remains to show that $ d^\alpha_t(\bv_t ) \leq d^\alpha_t(\bu) $.
This follows from Lemma \ref{monotonicity}, whose repeated application gives $\sigma_{t,i}( \bu) \geq \sigma_{t,i}( S_{1,\alpha}(\bu)) \geq \sigma_{t,i}(S_{2,\alpha} \circ S_{1,\alpha}(\bu)) \geq \ldots$ for each $t$.

To show the first inequality in \eqref{Salphaapprox}, we again invoke Lemma \ref{monotonicity}, in this case to infer that for each $t$,
\[    \sum_{i\geq 1} \abs{\sigma_{t,i}(\bS_\alpha(\bu))  - \sigma_{t,i}(\bu) }^2  \geq \sum_{i\geq 1}  \abs{\sigma_{t,i}(S_{t,\alpha}(\bu))  - \sigma_{t,i}(\bu) }^2 = \bigl( d_t^\alpha(\bu) \bigr)^2  \,. \]
Moreover, by Theorem \ref{thm:mirsky}, we have
\[  \norm{\sigma_{t}(\bS_\alpha(\bu))  - \sigma_{t}(\bu)} \leq \norm{\mathcal{M}_t(\bS_\alpha(\bu)) - \mathcal{M}_t(\bu)}_\HS  =  \norm{\bS_\alpha(\bu) - \bu} \,, \]
and taking the maximum over $t$ concludes the proof.
\end{proof}

\begin{remark}\label{salphaconv}
Let $\bu \in \Hcal$, then $ \sigma_t(\bu) \in \ell_2$, which implies that $d^\alpha_t(\bu) \to 0$ as $\alpha \to 0$. 
%To see this, first observe that $r_{t,\alpha} \to \infty$ as $\alpha\to 0$; we now rewrite
%$$  \norm{\sigma}^2 = \sum_{k=1}^\infty k (\sigma_{k}^2 - \sigma_{k+1}^2)  $$
%to obtain, for any $\alpha > 0$,
%$$ \alpha^2 r_{t,\alpha} \leq \sum_{k = r_{t,\alpha}}^\infty  (\sigma_{k}^2 - \sigma_{k+1}^2) \to 0 \quad  \text{as $\alpha \to 0$.} $$
Without further assumptions, however, this convergence can be arbitrarily slow.

\begin{enumerate}[{\rm (a)}]
\item If in addition $\sigma_{t}(\bu) \in \weakl{p}$ for a $p  \in (0,2)$ and for each $t$, we have 
\[   r_{t,\alpha} \lesssim \abs{\sigma_t(\bu)}_{\weakl{p}}^p \, \alpha^{-p} \,,\quad  \tau_{t,\alpha} \lesssim  \abs{\sigma_t(\bu)}^{p/2}_{\weakl{p}} \, \alpha^{1 - p /2 } \,, \]
see \cite{Devore:98}, and thus
\[   \norm{ \mathbf{S}_\alpha(\bu) -  \bu} \lesssim E\, \max_t \abs{\sigma_t(\bu)}^{p/2}_{\weakl{p}}  \, \alpha^{1-p/2} .   \]

\item If $\sigma_{t,k}(\bu) \leq C e^{- c k^\beta}$, $k\in\N$, with $C,c , \beta> 0$, then arguing similarly as in \cite[Section 7.4]{Devore:98}, we obtain
\[   r_{t,\alpha} \leq \bigl( c^{-1} \ln (C\alpha^{-1}) \bigr)^{\frac1\beta} \lesssim (1 + \abs{\ln\alpha})^{\frac1\beta}\,, \quad \tau_{t,\alpha} \lesssim (1 + \abs{\ln\alpha})^{\frac1{2\beta}} \,\alpha \,,  \]
and therefore
\begin{equation}\label{salphaconvexp}
    \norm{ \mathbf{S}_\alpha(\bu) -  \bu}   \lesssim  E\, (1 + \abs{\ln\alpha})^{\frac{1}{2\beta}}   \,  \alpha  \,.  
\end{equation}
\end{enumerate}
\end{remark}

\begin{remark}
It is well known that soft thresholding is closely related to convex optimization by proximal operator techniques. 
Note that the soft thresholding for hierarchical tensors can be written as
\begin{equation}
\mathbf{S}_{\alpha } =  S_{E, \alpha } \circ \cdots \circ S_{1, \alpha}      = 
\prox^{\alpha}_{J_E } \circ \cdots \circ \prox^{\alpha}_{J_1 }
\end{equation}
with $J_t := \norm{\mathcal{M}_t( \cdot) }_*$ for $t=1,\ldots,E$.
Thus in our setting, we do not have a characterisation of $\bS_\alpha$ by a single convex optimisation problem (as provided for $S_\alpha$ by Lemma \ref{lem:soft}), but still by a nested \emph{sequence} of convex optimization problems: one has $\bS_\alpha(\bu) = \tilde \bu_E$, where
\[   \tilde\bu_t := \argmin_{\bv \in \Hcal} \Bigl\{ J_t(\bv) + \frac1{2\alpha} \norm{\tilde\bu_{t-1} - \bv}^2 \Bigr\}  \,, \quad t = 1,\ldots, E\,,   \]
with $\tilde\bu_0 := \bu$.
\end{remark}

\section{Fixed-Point Iterations with Soft Thresholding}\label{sec:fpst}

In this section, we consider the combination of $\bS_\alpha$ with an arbitrary convergent fixed point iteration with a contractive mapping $\cF \colon  \cH \to \cH$, that is, there exists $\rho \in (0,1)$ such that
\begin{equation}\label{fpmapping}    \norm{\cF(\bv) - \cF(\bw)} \leq \rho \norm{\bv - \bw} \,,\quad \bv, \bw \in \cH \,.  \end{equation}
In the example of a linear operator equation $\Acal \bu = \mathbf{f}$ with elliptic $\Acal$, we may choose $\cF(\bv) = \bv - \mu (\Acal \bv - \mathbf{f})$ with a suitable scaling parameter $\mu >0$. A practical scheme for this particular case will be considered in detail in Section \ref{sec:aposteriori}.

Since $\bS_\alpha$ is non-expansive, the mapping $\bS_\alpha \circ \cF$ still yields a convergent fixed point iteration, but with a modified fixed point. 

\begin{lemma}\label{ualphaestimate}
Assuming \eqref{fpmapping}, let $\mathbf{u}^*$ be the unique fixed point of $\cF$.
Then for any $\alpha>0$, there exists a uniquely determined $\mathbf{u}^\alpha$ such that $\mathbf{u}^\alpha = \mathbf{S}_\alpha \bigl( \cF(\bu^\alpha) \bigr)$, which satisfies
\begin{equation}
\label{fpdist}
 ( 1+ \rho)^{-1} \norm{ \mathbf{S}_\alpha(\bu^*)-\bu^*}   \leq 
  \norm{ \mathbf{u}^\alpha  - \mathbf{u}^* } \leq (1-\rho)^{-1} \norm{ \mathbf{S}_\alpha(\bu^*)-\bu^*}   \,.
\end{equation}
Moreover, for any given $\mathbf{u}_0$, for $\mathbf{u}_{k+1} := \mathbf{S}_\alpha \bigl( \cF(\bu_k) \bigr)$ one has
$$  \norm{\mathbf{u}_k - \mathbf{u}^\alpha} \leq \rho^k \norm{\mathbf{u}_0 - \mathbf{u}^\alpha}  \,. $$
\end{lemma}

\begin{proof}
By the non-expansiveness of $\mathbf{S}_\alpha$, the operator 
$\cG := \mathbf{S}_\alpha \circ \cF$
is a contraction.
The existence and uniqueness of $\bu^\alpha$, as well as the stated properties of the iteration, thus follow from the Banach fixed point theorem. 
Let $\mathbf{e}^\alpha := \mathcal{G}(\bu^\alpha) - \mathcal{G}(\bu^*)$. Then since $\cG(\bu^\alpha) = \bu^\alpha$, one has
$$   \bu^\alpha - \bu^* = \mathcal{G}(\bu^*) - \bu^* + \mathbf{e}^\alpha \,.$$
Combining this with the observation
$$ \norm{\mathbf{e}^\alpha} =  \norm{\mathcal{G}(\bu^\alpha) - \mathcal{G}(\bu^*)} \leq \norm{\cF(\bu^\alpha) - \cF(\bu^*)} 
  \leq \rho \norm{\bu^\alpha-\bu^*}  \,,   $$
where we have again used that $\mathbf{S}_\alpha$ is non-expansive, yields
\[
    \norm{\bu^\alpha - \bu^*} \leq \norm{\mathcal{G}(\bu^*) - \bu^*} + \rho \norm{\bu^\alpha - \bu^*}  \,,\quad
    \norm{\mathcal{G}(\bu^*) - \bu^*} \leq \norm{\bu^\alpha - \bu^*} + \rho \norm{\bu^\alpha - \bu^*}  \,.   
\]
Finally, noting that $\mathcal{G}(\bu^*) = \mathbf{S}_\alpha(\bu^*)$ since $\mathcal{F}(\bu^*)=\bu^*$, we arrive at \eqref{fpdist}.
\end{proof}

Theorem \ref{ualphaestimate} tells us that if we keep the thresholding parameter $\alpha$ fixed, the thresholded Richardson iteration will converge, at the same rate $\rho$ as the unperturbed Richardson iteration, to a modified solution $\bu^\alpha$. Its distance to the true solution $\bu^*$ is uniformly proportional to $\norm{\mathbf{S}_\alpha(\bu^*)-\bu^*}$, that is, to the error of thresholding the exact solution.

\subsection{A Priori Choice of Thresholding Parameters}

In order to ensure convergence to $\bu^*$, instead of working with a fixed $\alpha$, we will instead consider the iteration
\begin{equation}\label{iteralphak}
\mathbf{u}_{k+1} =  \mathbf{S}_{\alpha_k}  \bigl( \cF(\bu_k) \bigr) \,,
\end{equation}
where we choose $\alpha_k$ with $\alpha_k \to 0$.
The central question is now how one can obtain a suitable such choice; clearly, if the $\alpha_k$ decrease too slowly, this will hamper the convergence of the iteration, whereas $\alpha_k$ that decrease too quickly may lead to very large tensor ranks of the iterates.

In principle, if the decay of the sequences $\sigma_t(\bu^*)$ is known, for instance $\sigma_t(\bu^*) \in \weakl{p}$, then Remark \ref{salphaconv} immediately gives us a choice of values for $\alpha_k$ that ensure convergence to $\bu^*$ with \emph{almost} the unperturbed rate $\rho$. To this end, observe that
\begin{align}  
 \norm{ \bu_{k+1} - \bu^* }  & \leq \norm{\bu_{k+1} - \bu^{\alpha_k} } + \norm{\bu^{\alpha_k} - \bu^*}  \leq \rho \norm{\bu_k - \bu^* } + ( 1 + \rho) \norm{\bu^{\alpha_k} - \bu^*}  \notag \\
    & \leq \rho \norm{\bu_{k} - \bu^* }  + \frac{1+\rho}{1-\rho}   \sum_{t} d^{\alpha_k}_t (\bu^*)  \,, \label{idealiter}
\end{align}
and based on Remark \ref{salphaconv} we can adjust $\alpha_k$ in every step so as to balance the decrease in the two terms on the right hand side of \eqref{idealiter}. Choices of $\alpha_k$ in \eqref{iteralphak} for the respective cases in Remark \ref{salphaconv} are given in the following proposition.

\begin{proposition}\label{idealconv}
Let $\sigma_t (\bu^*) \in \weakl{p}$, $t=1,\ldots,E$, for a $p\in(0,2)$, let $c_0>0$, and let $\bu_0 := 0$. Then for the choice $\alpha_k := (\rho^{k+1} c_0)^\frac{2}{2-p}$ in the iteration \eqref{iteralphak}, we have
\begin{equation}
 \norm{\bu_{k} - \bu^*} \leq \bigl( \norm{\bu^*} + C  E \max_t \abs{\sigma_t(\bu^*)}^{p/2}_{\weakl{p}}  \, k \bigr)\, \rho^k  \,,
\end{equation}
where $C$ depends on $\rho$, $p$, and $c_0$. Furthermore, for any $\tilde\rho > \rho$, with $\alpha_k := (\tilde\rho^{k+1} c_0)^\frac{2}{2-p}$, we have
\begin{equation}\label{idealconvest}
 \norm{\bu_{k} - \bu^*} \leq \Biggl( \norm{\bu^*} +  \frac{ C  E \max_t \abs{\sigma_t(\bu^*)}^{p/2}_{\weakl{p}}  \tilde{\rho}}{\tilde{\rho}-\rho} \Biggr)\, \tilde\rho^k 
\end{equation}
Under the same assumptions, but with the stronger condition $\sigma_{t,k}(\bu^*) \leq  C e^{- c k^\beta}$ with $C, c , \beta> 0$, for the choice $\alpha_k := \rho^{k+1} c_0$ we have
\begin{equation*}  
  \norm{\bu_{k} - \bu^*} \lesssim E k^{1 + {\frac{1}{2\beta}}} \, \rho^k   \,, 
\end{equation*}
and with $\alpha_k := \tilde\rho^{k+1} c_0$, we have instead
\begin{equation}  \label{idealconvexp}
 \norm{\bu_{k} - \bu^*} \lesssim E \,k^{ {\frac{1}{2\beta}}} \, \tilde\rho^k \,,  
\end{equation}
where the constant depends on $(\tilde\rho - \rho)^{-1}$.
\end{proposition}

\begin{proof}
Using \eqref{idealiter} and Remark \ref{salphaconv}(a), we obtain
\[    \norm{\bu_k - \bu^*} \leq \rho^k \norm{ \bu^* } + C_1 E \max_t \abs{\sigma_t(\bu^*)}^{\frac{p}2}_{\weakl{p}}   \sum_{i=0}^{k-1} \rho^{k-i-1} \alpha_i^{1-\frac{p}2} \,.  \]
Note that with our choice of $\alpha_k$, we have 
$\alpha_i^{1-\frac{p}2} = \rho^{i+1} c_0$,
which implies the first statement. For the second choice of $\alpha_k$, we obtain instead
\[  \norm{\bu_k - \bu^*}  \leq   \Bigl( \theta^k \norm{\bu^*} + C_1 E  \max_t \abs{\sigma_t(\bu^*)}^{\frac{p}2}_{\weakl{p}}  c_0 \sum_{i=0}^{k-1}  \theta^{k-1-i} \Bigr) \tilde\rho^k  \,, \quad \theta := \frac{\rho}{\tilde\rho}  < 1\,.   \]

Under the second set of assumptions, we proceed analogously based on Remark \ref{salphaconv}(b), which for $\alpha_k := \rho^{k+1} c_0$ yields
\[    \norm{\bu_k - \bu^*} \lesssim \rho^k  +  E \sum_{i=0}^{k-1} \rho^{k-i-1} (1 + \abs{\ln c_0} + i \abs{\ln \rho})^{\frac1{2\beta}} \rho^{i+1} \,.   \]
where the modified power of $k$ in the assertion thus arises due to the logarithmic term in \eqref{salphaconvexp}. For $\alpha_k := \tilde\rho^{k+1} c_0$, with $\theta$ as above,
\begin{equation*} 
 \norm{\bu_k - \bu^*} \lesssim \Bigl( \theta^k   +  E \sum_{i=0}^{k-1} \theta^{k-i-1} (1 + i \abs{\ln \tilde\rho})^{\frac1{2\beta}}  \Bigr) \tilde\rho^{i+1}   
  \lesssim E k^{\frac1{2\beta}} \tilde\rho^k \,. \qedhere
\end{equation*}
\end{proof}

In summary, in this idealized setting with full knowledge of the decay of $\norm{\mathbf{S}_\alpha(\bu^*) - \bu^*}$ with respect to $\alpha$, we can achieve convergence of the iteration with any asymptotic rate $\tilde \rho > \rho$.

\subsection{Rank Estimates}

We now give estimates for the ranks of the iterates that can arise in the course of the iteration, assuming that the $\alpha_k$ are chosen as in Proposition \ref{idealconv}.

For the proof we will use the following lemma, which is a direct adaptation of \cite[Lemma 5.1]{Cohen:01}, where the same argument was applied to hard thresholding of sequences; it was restated for soft thresholding of sequences (with the same proof) in \cite{DahlkeFornasierRaasch:12}.

\begin{lemma}\label{thresholdinglemma}
Let $\bv,\bw\in\Hcal$ and $\varepsilon >0$ such that $\norm{\bv-\bw}\leq \varepsilon$, and for $t\in \{1,\ldots,E\}$ let $\sigma_t(\bv)\in \weakl{p}$ for a $p\in(0,2)$. Then
$$   \rank_t \bigl(\mathbf{S}_\alpha(\bw)\bigr)  \leq \frac{4 \varepsilon^2}{\alpha^2} + C_p \abs{\sigma_t(\bv)}^p_{\weakl{p}} \alpha^{-p} \,.  $$
If $\sigma_{t,k}(\bv) \leq C e^{-c k^\beta}$ for $k\in\N$ with $C, c, \beta> 0$, then
$$   \rank_t \bigl(\mathbf{S}_\alpha(\bw)\bigr)  \leq \frac{4 \varepsilon^2}{\alpha^2} + \bigl( c^{-1} \ln (2C\alpha^{-1}) \bigr)^{1/\beta} \,.  $$
\end{lemma}

\begin{proof}
Note first that as an immediate consequence of Theorem \ref{thm:mirsky}, for each $t$ we have
\begin{equation*}   \norm{\sigma_t(\bv) - \sigma_t(\bw)} =  \Bigl( \sum_i \abs{\sigma_{t,i}(\bv)-\sigma_{t,i}(\bw)}^2\Bigr)^{\frac12} 
  \leq \norm{\mathcal{M}_t(\bv) - \mathcal{M}_t(\bw)}_\HS =  \norm{\bv - \bw} \leq \varepsilon \,. 
\end{equation*}
Furthermore,  Lemma \ref{monotonicity} yields $\rank_t(\mathbf{S}_\alpha(\bw)) \leq \rank_t(S_{t,\alpha}(\bw))$, and it thus suffices to estimate the latter.

The first inequality in the statement now follows with the same argument as in \cite{Cohen:01}, which we include for the convenience of the reader. We abbreviate $a:= \sigma_t(\bv)$, $b:= \sigma_t(\bw)$. Let $\mathcal{I}_1 := \{ i\colon b_i \geq \alpha, a_i > \alpha/2\}$ and $\mathcal{I}_2 := \{ i \colon b_i \geq \alpha, a_i \leq \alpha/2\}$.
Then 
\[   \Bigl( \frac\alpha2 \Bigr)^2 \#\mathcal{I}_2  \leq \sum_{i\in\mathcal{I}_2} \abs{a_i - b_i}^2 \leq  \varepsilon^2    \]
as well as
\[    \#\mathcal{I}_1 \leq \#\{ i\colon a_i > \alpha/2\} \leq C_p \abs{\sigma_t(\bv)}^p_{\weakl{p}} \alpha^{-p} \,,     \]
which proves the first statement, since
$  \rank_t(S_{t,\alpha}(\bw)) \leq \#\mathcal{I}_1 + \#\mathcal{I}_2 $.
To obtain the second inequality, we use Remark \ref{salphaconv}(b) to estimate $\#\mathcal{I}_1$ in an analogous way.
\end{proof}

\begin{theorem}\label{rankest1}
Let $\tilde\rho>\rho$, $\bu_0 := 0$, and $\varepsilon_k := \tilde\rho^k$.
If $\sigma_t (\bu^*) \in \weakl{p}$, $t=1,\ldots,E$, for a $p\in(0,2)$, then for the choice $\alpha_k := (\tilde\rho^{k+1} c_0)^\frac{2}{2-p}$ with $c_0>0$ in the iteration \eqref{iteralphak}, we have
\begin{equation}
  \norm{\bu_k -\bu^*} \lesssim d \varepsilon_k \,, \qquad  \max_{t=1,\ldots,E} \rank_t(\bu_k) \lesssim d^2 \varepsilon_k^{-\frac1s}\,,\; s = \textstyle\frac1p - \frac12 \,.
\end{equation}
Under the same assumptions, but with $\sigma_{t,k}(\bu^*) \leq C e^{- c k^\beta}$, $t=1,\ldots,E$, with $C, c, \beta> 0$, for the choice $\alpha_k := \tilde\rho^{k+1} c_0$ we have
\begin{equation}\label{rankest1exp}
  \norm{\bu_{k} - \bu^*} \lesssim d \bigl( 1+  \abs{\log \varepsilon_k} \bigr)^{\frac1{2\beta}}  \varepsilon_k \,, \qquad \max_{t=1,\ldots,E} \rank_t(\bu_k) \lesssim d^2 \bigl( 1+ \abs{\log \varepsilon_k}\bigr)^{\frac1\beta}  \,.
   \end{equation}
\end{theorem}

\begin{proof}
Recall that $E = 2d -3$.
The estimates $\norm{\bu_k - \bu^*}\lesssim E \varepsilon_k $ were shown in \eqref{idealconvest} and \eqref{idealconvexp} of Proposition \ref{idealconv}. To obtain the corresponding estimates for the ranks, we use Lemma \ref{thresholdinglemma}. Note that with $\bw_k := \cF(\bu_{k-1})$, we have $\bu_k = \mathbf{S}_{\alpha_{k-1}}(\bw_k)$ and, by contractivity of $\cF$,
\[   \norm{\bw_k - \bu^*} \leq  \rho \norm{\bu_{k-1} - \bu^*}  \lesssim E \varepsilon_{k-1}  \,.  \]
In the first case, for each $t$, Lemma \ref{thresholdinglemma} gives
$$   \rank_t(\mathbf{S}_{\alpha_{k-1}}(\bw_k)) \lesssim \frac{(E\varepsilon_k)^2}{\alpha_{k-1}^2} + \alpha_{k-1}^{-p} \lesssim E^2 \bigl( \tilde\rho^{1 -\frac{2}{2-p}}\bigr)^{2k} + \tilde\rho^{-\frac{2pk}{2-p}} \,. $$
Noting that $\bigl(\tilde\rho^{1 -\frac{2}{2-p}}\bigr)^{2k} = \tilde\rho^{-\frac{2pk}{2-p}} = \varepsilon_k^{-\frac1s}$, we obtain the first assertion.
In the second case we have $E\varepsilon_k / \alpha_k \lesssim E \bigl( 1 + \abs{\log \varepsilon_k} \bigr)^{\frac1{2\beta} }$, and the lemma thus yields
\begin{equation*}
   \rank_t(\mathbf{S}_\alpha(\bw_k))  \lesssim E^2 \bigl(1 + \abs{\log \varepsilon_k} \bigr)^{\frac1\beta } + \bigl(  1 + \abs{\log \tilde\rho^k c_0} \bigr)^{\frac1\beta}  
     \lesssim E^2 \bigl(1 + \abs{\log \varepsilon_k} \bigr)^{\frac1\beta} \,.  \qedhere  
\end{equation*}
\end{proof}

\section{A Posteriori Choice of Parameters}\label{sec:aposteriori}

The results of the previous section lead to the question whether the results in Theorem \ref{rankest1} can still be recovered when a priori knowledge of the decay of the sequences $\sigma_t(\bu^*)$ is not available.
We thus now consider a modified scheme that adjusts the $\alpha_k$ automatically without using such information on $\bu^*$, but still yields quasi-optimal rank estimates as in Theorem \ref{rankest1} for both cases considered there.

The design of such a method is more problem-specific than the general considerations in the previous section, and here we thus restrict ourselves to linear operator equations
$ \mathcal{A} \bu = \mathbf{f} $
with $\mathcal{A}$ symmetric and elliptic.
We assume to have $\gamma, \Gamma >0$ such that 
\begin{equation}\label{gammaass}
  \gamma \norm{\bv}^2 \leq \langle \Acal \bv,\bv\rangle\leq \Gamma \norm{\bv}^2\,,  \quad \bv\in\Hcal\,,
\end{equation}
that is, the spectrum of $\Acal$ is contained in $[\gamma,\Gamma]$ and $\kappa := \gamma^{-1}\Gamma$ is an upper bound for the condition number of $\Acal$. The choice $\mu := 2 / (\gamma + \Gamma)$ then yields 
\begin{equation}\label{contr}
   \norm{\id - \mu \mathcal{A}} \leq \frac{\kappa - 1}{\kappa + 1} =: \rho < 1 \,,
\end{equation}
and the results of the previous section apply with $\cF(\bv) := \bv - \mu (\mathcal{A} \bv - \mathbf{f})$.

\begin{algorithm}[!ht]
\caption{$\quad \mathbf{u}_\varepsilon = \solve( \Acal, \mathbf{f}; \varepsilon)$} 
\begin{algorithmic}[1]
\Require $\gamma$, $\Gamma$ as in \eqref{gammaass}, $\mu$, $\rho$ as in \eqref{contr}, arbitrary $\nu, \theta \in (0,1)$ and $\alpha_0\geq E^{-1}\mu \norm{\mathbf{f}}$.
\Ensure $\mathbf{u}_\varepsilon$ satisfying $\norm{\mathbf{u}_\varepsilon - \mathbf{u}^*}\leq \varepsilon$.
\State $\mathbf{u}_0 := 0$, $\mathbf{r}_{0} := - \mathbf{f}$, $k:=0$
\While{$\norm{\mathbf{r}_k} > \gamma \varepsilon$}
\State $\mathbf{u}_{k+1} = \mathbf{S}_{\alpha_k} \bigl(\mathbf{u}_{k} - \mu \mathbf{r}_{k} \bigr)$
\State $\mathbf{r}_{k+1} := \mathcal{A}\mathbf{u}_{k+1} -\mathbf{f}$
\If{$\displaystyle  \norm{\bu_{k+1} - \bu_k} \leq \frac{(1 - \rho) \nu}{ \Gamma \rho} \norm{\mathbf{r}_{k+1}} $}\label{condline}
\State $\alpha_{k+1} := \theta \alpha_k$
\Else
\State $\alpha_{k+1} := \alpha_k$
\EndIf 
\State $k \gets k+1$
\EndWhile
\State $\mathbf{u}_\varepsilon := \mathbf{u}_k$ 
\end{algorithmic}
\label{alg:stsolve}
\end{algorithm}

In order to be able to obtain estimates for the ranks of iterates as in Theorem \ref{rankest1}, the method in Algorithm \ref{alg:stsolve} is constructed such that whenever $\alpha_k$ is decreased in the iteration,
\begin{equation}\label{cond0intro}   \norm{\bu_{k+1} - \bu^*} \leq C \norm{\bu^{\alpha_k} - \bu^*}  
\end{equation}
holds with some fixed constant $C>1$. It will be established in what follows that the validity of such an estimate ensures that $\alpha_k$ never becomes too small in relation to the corresponding current error $\norm{\bu_k - \bu^*}$.
A bound of the form \eqref{cond0intro} is ensured by the condition in line \ref{condline} of Algorithm \ref{alg:stsolve}, which is explained in more detail in the proof of Theorem \ref{rankest2} below.

Note that Algorithm \ref{alg:stsolve} only requires -- besides a hierarchical tensor representation of $\mathbf{f}$ and the action of $\Acal$ on such representations -- bounds $\gamma$, $\Gamma$ on the spectrum of $\Acal$, certain quantities derived from these, as well as constants that can be adjusted arbitrarily.
The following is the main result of this work.

\begin{theorem}\label{rankest2}
Algorithm \ref{alg:stsolve} produces $\bu_\varepsilon$ with $\norm{\bu_\varepsilon - \bu^*} \leq \varepsilon$ in finitely many steps.
Furthermore, if $\sigma_t (\bu^*) \in \weakl{p}$, $t=1,\ldots,E$, for a $p\in(0,2)$, then there exists $\tilde\rho \in (0,1)$ such that with $\varepsilon_k := \tilde\rho^k$, the iterates satisfy
\begin{equation}\label{thmwlp}
  \norm{\bu_k -\bu^*} \lesssim d \varepsilon_k \,, \qquad  \max_{t=1,\ldots,E}\rank_t(\bu_k) \lesssim d^2 \max_{\tau=1,\ldots,E} \abs{\sigma_\tau(\bu^*)}_{\weakl{p}}^{\frac1s} \varepsilon_k^{-\frac1s}\,,\quad s = \textstyle\frac1p - \frac12 \,.
\end{equation}
If $\sigma_{t,k}(\bu^*) \leq C e^{- c k^\beta}$, $t=1,\ldots,E$, with $C, c, \beta> 0$, then the analogous statement holds with
\begin{equation}\label{thmexp}
  \norm{\bu_{k} - \bu^*} \lesssim d \varepsilon_k \,, \qquad \max_{t=1,\ldots,E}\rank_t(\bu_k) \lesssim d^2  \bigl( 1+ \abs{\ln \varepsilon_{k}} \bigr)^{\frac1\beta }  \,.
   \end{equation}
\end{theorem}

In the proof we will use the following technical lemma, which limits the decay of the soft thresholding error as the thresholding parameter is decreased. 

\begin{lemma}\label{invbound}
Let $\bv\neq 0$, then
$ d^{\alpha}_t(\bv) \leq \theta^{-1} d^{\theta\alpha}_t (\bv) $, $t=1,\ldots,E$,
for all  $\alpha >0$, $\theta \in (0,1)$.
\end{lemma}

\begin{proof}
For the proof, we omit the dependence of quantities on $\bv$. We clearly have $r_{t, \theta\alpha} \geq r_{t,\alpha}$ and $\tau_{t, \theta\alpha} \leq \tau_{t,\alpha}$.
Furthermore,  $ \tau_{t,\alpha}^2 - \tau_{t,\theta\alpha}^2 \leq \alpha^2 ( r_{t, \theta\alpha} - r_{t, \alpha}) $, and consequently,
\[   \biggl( \frac{d^{\alpha}_t}{d^{\theta\alpha}_t  }\biggr)^2  = \frac{\alpha^2 r_{t,\alpha}  + \tau_{t,\alpha}^2 }{(\theta \alpha)^2 r_{t,\theta\alpha} + \tau_{t,\theta\alpha}^2 }   \leq    \frac{\alpha^2 r_{t,\alpha}  + \tau^2_{t,\theta\alpha} + \alpha^2 ( r_{t,\theta\alpha} - r_{t,\alpha}) }{(\theta\alpha)^2 r_{t,\theta\alpha} + \tau_{t,\theta\alpha}} \leq \theta^{-2}  \,.  \qedhere  \]
\end{proof} 

\begin{proof}[Proof of Theorem \ref{rankest2}]
\emph{Step 1:}
We first show that the condition
\begin{equation}\label{stsolve_condition}
     \norm{\bu_{k+1} - \bu_k} \leq \frac{(1 - \rho) \nu}{ \Gamma \rho} \norm{\mathbf{r}_{k+1}}
\end{equation}
in line \ref{condline} of the algorithm is always satisfied after a finite number of steps.

The combination of the first inequality of \eqref{Salphaapprox} in Lemma \ref{lmm:st} and the first inequality in \eqref{fpdist} of Lemma \ref{ualphaestimate} shows
\[    ( 1+ \rho)^{-1} \max_t d^\alpha_t(\bu^*)   \leq \norm{\bu^\alpha - \bu^*}   \,.  \]
Thus we always have $\norm{\bu^\alpha - \bu^*} > 0$ if $\alpha>0$, unless $\bu^* = 0$ and hence $\mathbf{f} = 0$. In the latter case, however, the algorithm stops immediately, and we can thus assume that $\bu^\alpha \neq \bu^*$ for any positive $\alpha$.

If $\alpha_k = \ldots = \alpha_0$, we have on the one hand
\begin{equation}\label{condleft0}
  \norm{\bu_{k+1} - \bu_k} \leq \norm{\bu_{k+1} - \bu^{\alpha_0} } + \norm{\bu_k - \bu^{\alpha_0}} \leq \rho^k ( 1 + \rho ) \norm{\bu_0 - \bu^{\alpha_0}} \,,
\end{equation}
and on the other hand, we similarly obtain
\begin{align}\label{condright0}
  \gamma^{-1} \norm{\br_{k+1}} \geq  \norm{\bu_{k+1} - \bu^*}  & \geq \norm{\bu^{\alpha_0} - \bu^* } - \rho^{k+1} \norm{\bu_0  - \bu^{\alpha_0} }  \notag \\
   & \geq (1 - \rho^{k+1}) \norm{\bu^{\alpha_0} - \bu^*} - \rho^{k+1} \norm{\bu_0 - \bu^*}      \,.
\end{align}
Thus  the right hand side in \eqref{condleft0} converges to zero, whereas the right hand side in \eqref{condright0} is bounded away from zero for sufficiently large $k$.
Hence \eqref{stsolve_condition} holds with $k = J_0 -1$ for some $J_0 \in \N$, and we assume this to be the minimum integer with this property. The thresholding parameter is then decreased for the following iteration, that is, $\alpha_{J_0} = \theta \alpha_{J_0 -1} = \theta\alpha_{0}$. 
As in \eqref{idealiter}, for $k < J_0$ we obtain
\begin{equation}\label{iterbounded}
   \norm{ \bu_{k+1} - \bu^* }  \leq  \rho^{k+1} \norm{\bu_0 - \bu^*}  + (1 + \rho^{k+1}) \norm{\bu^{\alpha_0} - \bu^*} \,.
\end{equation}
The same arguments then apply with $\alpha_{0}$ replaced by $\alpha_{J_0}$ and $\bu_0$ by $\bu_{J_0}$.
Thus $\alpha_k$ will always be decreased after a finite number of steps.

\emph{Step 2:} To show convergence of the $\bu_k$, we first observe that by our requirement that $\alpha_0 \geq E^{-1} \mu \norm{\mathbf{f}}$ (which is in fact not essential for the execution of the iteration), we actually have $\bu_1=0$ and hence $\bu^{\alpha_0} = 0$, implying also $J_0=1$. In particular,
\begin{equation}\label{prop0}
   \norm{ \bu_{n} - \bu^* } \leq \norm{  \bu^* - \bu^{\alpha_0} }  \,, \quad 0 \leq n \leq J_0 = 1\,. 
\end{equation}
We next investigate the implications of the condition \eqref{stsolve_condition} for the further iterates.
Note first that 
\begin{equation}\label{invtriangle}  \norm{\bu_{k+1} - \bu^*} - \norm{\bu_{k+1} - \bu^{\alpha_{k}}} \leq \norm{\bu^* - \bu^{\alpha_k}}  \,. \end{equation}
The standard error estimate for contractive mappings, combined with  \eqref{stsolve_condition} and \eqref{gammaass}, gives 
\begin{equation}\label{nueq}
   \norm{\bu_{k+1} - \bu^{\alpha_k}} \leq \frac{\rho}{1-\rho} \norm{\bu_{k+1} - \bu_k} \leq \frac{\nu}{\Gamma}\norm{\mathbf{r}_{k+1}} \leq   \nu \norm{\bu_{k+1} - \bu^*}  \,.
\end{equation}
Inserting the latter into \eqref{invtriangle}, we thus obtain
\begin{equation}\label{cond}    \norm{\bu_{k+1} - \bu^*} \leq (1 - \nu)^{-1} \bigl(  \norm{\bu_{k+1} - \bu^*} - \norm{\bu_{k+1} - \bu^{\alpha_k}}   \bigr)  \leq (1 - \nu)^{-1}  \norm{\bu^* - \bu^{\alpha_k}}    \,.   
\end{equation}

We introduce the following notation that groups iterates according to the corresponding values of $\alpha_k$: for each $i \in \N_0$, let $\eta_{i} := \theta^i \alpha_0$. With $\bu_{0,0} = \bu_0 =  0$ and $\br_{0,0} = - \mathbf{f}$, iterates are produced according to 
\begin{equation}\label{proofabbrev}   \bw_{i,j+1} := \bu_{i,j} - \mu \br_{i,j} , \quad  \bu_{i,j + 1} := \bS_{\alpha_{i,j}} (\bw_{i,j + 1}),  \quad  \br_{i,j+1} := \mathcal{A} \bu_{i,j+1} - \mathbf{f} \,,\end{equation}
where the index $i$ is increased each time that condition \eqref{stsolve_condition} is satisfied.
For each $i$, consistently with the previous definition of $J_0$, we define $J_i$ as the last index of an iterate produced with the value $\eta_{i}$,  which means that $\bu_{i+1,0} = \bu_{i,J_i}$ and, as a consequence of \eqref{cond},
\begin{equation}\label{propinit}
   \norm{  \bu_{i+1, 0} - \bu^* } = \norm{ \bu_{i, J_i} - \bu^* } \leq (1-\nu)^{-1} \norm{\bu^* - \bu^{\eta_{i}} }  \,, \quad i \geq 0 \,.
\end{equation}
For $i  \geq 0$ and $j=0,\ldots,J_i$, with \eqref{propinit} we obtain
\begin{align}\label{propi}
   \norm{\bu_{i,j} - \bu^*} & \leq \norm{\bu_{i,j} - \bu^{\eta_i} } + \norm{ \bu^{\eta_i} - \bu^*}   \notag  \\
%     & \leq \rho^j \norm{ \bu_{i,0} - \bu^{\eta_i}} + \norm{ \bu^{\eta_i} - \bu^*} \notag  \\
      & \leq \rho^j \norm{\bu_{i,0} - \bu^*} + (1 + \rho^j) \norm{ \bu^{\eta_i} - \bu^*}  \notag  \\
      & \leq  (1 - \nu)^{-1} \rho^j \norm{\bu^{\eta_{i-1}} - \bu^*} + (1 + \rho^j) \norm{ \bu^{\eta_i} - \bu^*}  \,,
\end{align}
where we have used \eqref{prop0} in the case $i=0$.
By Remark \ref{salphaconv} this implies in particular that, in our original notation, $\bu_k \to \bu^*$.

\emph{Step 3:} Our next aim is to estimate the values of $J_i$. We have already established that $J_0 = 1$.
In order to estimate $J_i$ for $i>0$, we use \eqref{condleft0} and \eqref{condright0} to obtain
\begin{equation}
  \frac{  \norm{\bu_{i,j+1} -  \bu_{i,j}} }{  \norm{\br_{i,j+1} } } \leq \frac{   \rho^j \gamma^{-1} (1+\rho)  \norm{ \bu_{i,0} - \bu^{\eta_i}}  }{   \norm{\bu^{\eta_i} - \bu^*}  - \rho^{j+1} \norm{\bu_{i,0} - \bu^{\eta_i}}  }   
\end{equation}
for $j$ sufficiently large.
Thus \eqref{stsolve_condition} follows if the two conditions
\[
   \rho^j \gamma^{-1} (1+\rho) \frac{\norm{\bu_{i,0} - \bu^{\eta_i}}}{\norm{\bu^{\eta_i} - \bu^*}} \leq \frac{(1-\rho) \nu}{2 \Gamma\rho} \,,\quad
      \rho^{j+1} \frac{\norm{\bu_{i,0} - \bu^{\eta_i}}}{\norm{\bu^{\eta_i} - \bu^*}}  \leq \frac{1}{2}
\]
hold. These are guaranteed if
\begin{equation}\label{Jibound} 
   j \geq \abs{\ln \rho}^{-1} \biggl( C(\gamma,\Gamma,\nu) +  \ln  \frac{\norm{\bu_{i,0} - \bu^{\eta_i}}}{\norm{\bu^{\eta_i} - \bu^*}}  \biggr)  
\end{equation}
with some constant $C(\gamma,\Gamma,\nu)\geq0$.  By \eqref{propinit},
\[   \ln \frac{\norm{\bu_{i,0} - \bu^{\eta_i}}}{\norm{\bu^{\eta_i} - \bu^*}} \leq \ln \biggl( 1 + (1- \nu)^{-1} \frac{\norm{\bu^{\eta_{i-1}} - \bu^*}}{\norm{\bu^{\eta_i} - \bu^*}} \biggr) \,,  \]
and by Lemma \ref{lmm:st} and Lemma \ref{ualphaestimate},
\[    \frac{\norm{\bu^{\eta_{i-1}} - \bu^*}}{\norm{\bu^{\eta_i} - \bu^*}} \leq \frac{(1+\rho) \norm{\bS_{\eta_{i-1}}(\bu^*) - \bu^*} }{(1-\rho) \norm{\bS_{\eta_{i}}(\bu^*) - \bu^*}  }  \leq \frac{(1+\rho)}{(1-\rho)}  \sum_{t=1}^E \frac{d^{\theta^{i-1}\alpha_0}_t(\bu^*) }{ d^{\theta^i \alpha_0}_t(\bu^*) }   \leq  \frac{(1+\rho) E}{(1-\rho) \theta}   \,,  \]
where we have used that as a consequence of Lemma \ref{invbound}, the quotients $ {d^{\theta^{i-1}\alpha_0}_t(\bu^*) }/{ d^{\theta^i \alpha_0}_t(\bu^*) } $ remain uniformly bounded by $\theta^{-1}$. Putting this together with \eqref{Jibound}, we thus obtain
\[   J_i \lesssim \ln ( d) \,,  \]
with a uniform constant depending on $\gamma$, $\Gamma$, $\nu$, and $\theta$. In view of \eqref{propi}, this implies that $\bu_k$ converges to $\bu^*$ at a linear rate in cases \eqref{thmwlp} and \eqref{thmexp}.

\emph{Step 4:} In order to establish rank estimates, we need to bound the errors of $\bw_{i,j}$ as defined in \eqref{proofabbrev} for each $i \geq 0$ and $0 < j \leq J_i$. Since $\bu_0 = \bu^{\eta_0} = 0$ by our choice of $\alpha_0$, for $i=0$ we obtain
\begin{equation}\label{w0}  \norm{  \bw_{0,j}  - \bu^* } \leq \rho \norm{ \bu_{0,j-1} - \bu^* } = \rho  \norm{  \bu^{\eta_0}  - \bu^* } \,, \quad j = J_0 = 1   \,,
\end{equation}
and for $i>0$ and $j>0$, by \eqref{propi},
\begin{equation}\label{wi}
    \norm{ \bw_{i,j} - \bu^* }  \leq  \rho \norm{ \bu_{i,j-1} - \bu^*}  
      \leq  ( 1 - \nu )^{-1} \rho^{j} \norm{\bu^{\eta_{i-1}} - \bu^*} + \rho (1 + \rho^{j-1}) \norm{ \bu^{\eta_i} - \bu^*}.
\end{equation}
By Lemma \ref{thresholdinglemma}, with $M := \max_t \abs{\sigma_t(\bu^*)}_{\weakl{p}}$, for all $t$ we have
\begin{equation*}
   \rank_t (\bu_{i,j} )  \lesssim  \frac{ \norm{ \bw_{i,j} - \bu^* }^2 }{ \eta_i^2} + f(\eta_i) \,,\quad f(\eta_i) := \begin{cases}
    M^p \eta_i^{-p}, & \text{ in case \eqref{thmwlp}, }\\
    (1 + \abs{\ln\eta_i})^{\frac1\beta} ,& \text{ in case \eqref{thmexp}.}
   \end{cases}
\end{equation*}
Note that this also covers $\bu_{i,0}$ for $i \geq 0$, since $\bu_{i,0} = \bu_{i-1,J_{i-1}}$ for $i>0$ and $\bu_{0,0} = 0$.

In case \eqref{thmwlp}, as a consequence of \eqref{prop0}, \eqref{propi}, \eqref{w0}, \eqref{wi}, with Remark \ref{salphaconv}(a) we obtain
\[    \norm{\bw_{i,j} - \bu^*}   \; \lesssim \; E  M^{\frac{p}2} \eta_i^{1 - \frac{p}2}  \]
for all respective $i,j$, and consequently
\[        \rank_t (\bu_{i,j} )   \lesssim E^2 M^p \eta_i^{2-p-2} + M^p \eta_i^{-p} = (1 + E^2) M^p \eta_i^{-p}\,.   \]
By the same argument, we also have $\norm{\bu_{i,j} - \bu^*} \lesssim E M^{\frac{p}2} \eta_i^{1 - \frac{p}2}$. Setting $\varepsilon_{i,j} := M^{\frac{p}2} \eta_i^{1 - \frac{p}2}$, we thus have $ \norm{\bu_{i,j} - \bu^*} \lesssim E \varepsilon_{i,j}$
as well as
\[    \rank_t (\bu_{i,j} )   \lesssim (1+E^2)  M^{\frac{1}s} \varepsilon_{i,j}^{-\frac1s}\,,\quad s = \frac1p - \frac12 \,.   \]
This completes the proof of \eqref{thmwlp}. 
In the case \eqref{thmexp}, Remark \ref{salphaconv}(b) yields, expanding $\eta_i = \theta^i \alpha_0$,
\[    \norm{\bu_{i,j} - \bu^*}, \; \norm{\bw_{i,j} - \bu^*}   \; \lesssim \; E ( 1 + i \abs{\ln \theta} )^{\frac1{2\beta}} \theta^i \,,  \]
and hence
\[     \rank_t (\bu_{i,j} )   \lesssim E^2  ( 1 + i \abs{\ln \theta} )^{\frac1\beta }  + ( 1 + i \abs{\ln \theta})^{\frac1\beta} \,.     \]
We choose $\tilde\theta\in (\theta, 1)$ and set $\varepsilon_{i,j} := \tilde\theta^i$ to obtain $  \norm{\bu_{i,j} - \bu^*} \lesssim E \varepsilon_{i,j} $
and 
\begin{equation*}     \rank_t (\bu_{i,j} )    \lesssim E^2 \biggl[ \biggl(  1+  \frac{\abs{\ln \theta}}{ \abs{\ln \tilde\theta}}  \abs{\ln \tilde\theta^i} \biggr) \biggr]^{\frac1\beta }     \lesssim E^2 \bigl( 1+ \abs{\ln \varepsilon_{i,j}} \bigr)^{\frac1\beta}  \,.   
\end{equation*}
This completes the proof of \eqref{thmexp}.
\end{proof}

\begin{remark}
The above algorithm is universal in the sense that it does not require knowledge of the decay of the $\sigma_t(\bu^*)$, but we still obtain the same quasi-optimal rank estimates as with $\alpha_k$ prescribed a priori as in Theorem \ref{rankest1}.
Note that in \eqref{thmexp}, we have absorbed the additional logarithmic factor that is present in the error bound in \eqref{rankest1exp} by comparing to a slighly slower rate of linear convergence, but the estimates are essentially of the same type.
\end{remark}

\begin{remark}
For the effective convergence rate $\tilde\rho$ in the statement of Theorem \ref{rankest2}, as can be seen from the proof (in particular from the estimates for the $J_i$), one has an estimate from above of the form $\tilde\rho \leq \hat\rho^{\frac1{\log d}} < 1$, where $\hat \rho$ does not explicitly depend on $d$ (although it may still depend on $d$ through other quantities such as $\gamma$, $\Gamma$). Consequently, combining this with the statements in \eqref{thmwlp} and \eqref{thmexp}, we generally have to expect that the number of iterations required to ensure $\norm{\bu_k - \bu^*} \leq \varepsilon$ scales like $\abs{\log \hat \rho}^{-1} \bigl( (\log \varepsilon + \log d ) \log d\bigr)$.
\end{remark}

\subsection{Inexact Evaluation of Residuals}

We finally consider a perturbed version of Algorithm \ref{alg:stsolve} where residuals are no longer evaluated exactly, but only up to a certain relative error.
We assume that for each given $\bv$ and $\delta > 0$, we can produce $\mathbf{r}$ such that $\norm{\mathbf{r} - (\Acal \bv - \mathbf{f})} \leq \delta$. 

We will show below that for our purposes it suffices to ensure a certain relative error for each $\br_k$ computed in Algorithm \ref{alg:stsolve}, more precisely, to adjust $\delta$ for each $k$ such that 
\[  \norm{\mathbf{r}_k - (\Acal \bu_k - \mathbf{f})} \leq \min\{ \tau_1 \norm{\mathbf{r}_k},  \tau_2 \mu^{-1} \norm{\bu_{k+1} - \bu_{k}} \} \,, \]
with suitable $\tau_1, \tau_2 > 0$.
This can be achieved by simply decreasing the value of $\delta$ and recomputing $\br_k$ (and the resulting $\bu_{k+1}$) until $\delta \leq \min\{ \tau_1 \norm{\mathbf{r}_k},  \tau_2 \mu^{-1}\norm{\bu_{k+1} - \bu_{k}} \}$ is satisfied. With such a choice of $\delta$, we then have in particular
\begin{equation}\label{residualrelative}
  (1-\tau_1) \norm{\mathbf{r}_k} \leq \norm{\Acal \bu_k - \mathbf{f}} \leq (1+\tau_1) \norm{\mathbf{r}_k} \,.
\end{equation}
Our scheme can be regarded as an extension of the residual evaluation strategy used in \cite{GHS} in the context of an adaptive wavelet scheme, where the residual error is controlled relative to the norm of the computed residual.
Note that in our algorithm, the error tolerance $\delta_k$ used for each computed $\br_k$ is adjusted twice: first in line \ref{residual1line} to ensure the accuracy with respect to $\norm{\br_k}$, and possibly a second time in line \ref{residual2line} (after incrementing $k$) to ensure the accuracy with respect to $\norm{\bu_{k+1} - \bu_k}$.

The analysis of the resulting modified Algorithm \ref{alg:iestsolve} follows the same lines as the proof of Theorem \ref{rankest2}, and we obtain the same statements with modified constants.
We do not restate the full proof, but instead indicate how the central estimates are modified. For a given iterate $\bu_k$, we now denote the exact residual by $\mathbf{\bar r}_k := \Acal \bu_k - \mathbf{f}$ and the computed residual by $\mathbf{r}_k$. 

We first consider the influence of the perturbation on the iteration without thresholding, but with inexact residual, for which we obtain
\begin{equation*}
   \norm{ (  \bu_k - \mu \br_k  )  - \bu^* } \leq \norm{ (\bu_k - \mu \mathbf{\bar r}_k) - \bu^* } + \mu \norm{\br_k - \mathbf{\bar r}_k } \leq 
     \rho \norm{\bu_k - \bu^*} + \tau_1 \mu \norm{\br_k}    \,. 
\end{equation*}
Using $\norm{\br_k } \leq (1 - \tau_1)^{-1} \norm{\mathbf{\bar r}_k } \leq ( 1- \tau_1)^{-1} \Gamma \norm{\bu_k - \bu^*}$ and $\mu \Gamma\leq 2$,
\begin{equation}\label{perturbedcontr}
   \norm{ (  \bu_k - \mu \br_k  )  - \bu^* } \leq  \biggl( \rho + \frac{ 2 \tau_1 } { 1 - \tau_1}  \biggr) \norm{\bu_k - \bu^*} \,,
\end{equation}
which yields a contraction provided that $\tau_1 < ( 3 - \rho)^{-1} (1 - \rho)$.  However, in \eqref{w0} and \eqref{wi}, where the bound \eqref{perturbedcontr} is required, $\tau_1 < 1$ is in fact sufficient.

We now turn to the contractivity of the iteration with thresholding. Note that 
\begin{equation}\label{iecontr}
   \norm{ \bu_{k+1} - \bu_k } \leq
     \norm{ \bu_{k + 1} - \bS_{\alpha_k}(\bu_k - \mu \mathbf{\bar r}_k) }  +
       \norm{\bS_{\alpha_k}(\bu_k - \mu \mathbf{\bar r}_k) - \bu^{\alpha_k} }  +  \norm{\bu^{\alpha_k} - \bu_k}  \,,
\end{equation}
where by non-expansiveness of $\bS_{\alpha_k}$,
\begin{equation*}
  \norm{ \bu_{k + 1}  -  \bS_{\alpha_k}(\bu_k - \mu \mathbf{\bar r}_k) } \leq \mu \norm{\br_k - \mathbf{\bar r}_k} \,,
\end{equation*}
and thus, since $\mu \norm{\br_k - \mathbf{\bar r}_k} \leq \tau_2 \norm{\bu_{k+1} - \bu_{k}}$ by our construction, \eqref{iecontr} gives
\begin{equation*}
  \norm{\bu_{k+1} - \bu_k} \leq \frac{1+\rho}{1- \tau_2} \norm{ \bu_k - \bu^{\alpha_k} } \,.
\end{equation*}
As a consequence,
\begin{align}
  \norm{\bu_{k+1} - \bu^{\alpha_k} } & \leq  \norm{ \bS_{\alpha_k}(\bu_k - \mu \mathbf{\bar r}_k)  -  \bu^{\alpha_k} }  + \norm{\bS_{\alpha_k} ( \bu_k - \mu \br_k) - \bS_{\alpha_k}(\bu_k - \mu \mathbf{\bar r}_k )}  \notag  \\[2pt]
  & \leq \rho \norm{\bu_k - \bu^{\alpha_k}} + \mu \norm{\br_k - \mathbf{\bar r}_k}   \notag \\
  & \leq \hat\rho(\tau_2) \norm{\bu_k - \bu^{\alpha_k}} \,, \qquad \hat\rho(\tau_2) := \rho + \frac{(1+\rho)\tau_2}{1-\tau_2} \,, \label{mod_contr}
\end{align}
where $\hat\rho(\tau_2) <1$ holds precisely when $\tau_2 < \frac12 (1-\rho)$; in other words, the perturbed fixed point iteration then has the same contractivity property with a modified constant. 

Furthermore, we show next that the validity of the modified condition
\begin{equation} \label{mod_stsolve_condition}
  \norm{\bu_{k+1} - \bu_{k}} \leq B \norm{\mathbf{r}_{k+1}} \,, \quad 
    B :=  \frac{(1 - \rho)(1-\tau_1) \nu}{  (1+\tau_2) (\rho  + (1-\tau_2)^{-1}(1+\rho)\tau_2) \Gamma } \,,
\end{equation}
in Algorithm \ref{alg:iestsolve} still implies that the corresponding iterates satisfy \eqref{cond}.

\begin{algorithm}[!h]
\caption{$\quad \mathbf{u}_\varepsilon = \iesolve(\Acal, \mathbf{f};  \varepsilon)$} 
\begin{algorithmic}[1]
\Require $\mu$, $\rho$ as in \eqref{contr}, arbitrary $\nu, \theta, \omega \in (0,1)$, $\alpha_0 \geq E^{-1} \mu \norm{\mathbf{f}}$, 
\Statex $\tau_1 \in (0, 1)$, $\tau_2 \in (0, \frac12 (1-\rho))$,
$B$ as in \eqref{mod_stsolve_condition}, $D$ as in \eqref{const2}.
\Ensure $\mathbf{u}_\varepsilon$ satisfying $\norm{\mathbf{u}_\varepsilon - \mathbf{u}^*}\leq \varepsilon$.
\State $\mathbf{u}_0 := 0$, $\mathbf{r}_{0} := - \mathbf{f}$
\State $k:= 0$, $\delta_0 :=  \tau_1 \norm{\br_0} $
\While{$\norm{\mathbf{r}_k}  + \delta_k >  \gamma \varepsilon$}
\State $\mathbf{u}_{k+1} := \mathbf{S}_{\alpha_k} \bigl(\mathbf{u}_{k} - \mu \mathbf{r}_{k} \bigr)$
\While{$\delta_k >  \tau_2 \mu^{-1}\norm{\mathbf{u}_{k+1} - \bu_k}$  ~ $\wedge$ ~ $\displaystyle \delta_k > D \norm{\br_k} $ with $D$ as in \eqref{const2}}   \label{secondetacondline}
\State $\delta_{k} \gets \omega \delta_{k}$
\State compute $\br_{k}$ such that $\norm{\mathbf{r}_{k} -  (\mathcal{A}\mathbf{u}_{k} -\mathbf{f})} \leq \delta_{k}$ \label{residual2line}
\State $\mathbf{u}_{k+1} \gets \mathbf{S}_{\alpha_k} \bigl(\mathbf{u}_{k} - \mu \mathbf{r}_{k} \bigr)$
\EndWhile
\State $\delta_{k+1} :=  \omega^{-1} \delta_{k}$
\Repeat
\State $\delta_{k+1} \gets \omega \delta_{k+1}$
\State compute $\br_{k+1}$ such that $\norm{\mathbf{r}_{k+1} -  (\mathcal{A}\mathbf{u}_{k+1} -\mathbf{f})} \leq \delta_{k+1}$ \label{residual1line}
\If{$\norm{\br_{k+1}} + \delta_{k+1} \leq  \gamma\varepsilon$} \label{residualexitline}
\State set $\mathbf{u}_\varepsilon := \mathbf{u}_{k+1}$ and stop
\EndIf
\Until{ $\delta_{k+1} \leq \tau_1 \norm{\br_{k+1}}$ }
\If{$\displaystyle  \norm{\bu_{k+1} - \bu_{k}} \leq B \norm{\mathbf{r}_{k+1}} $ with $B$ as in \eqref{mod_stsolve_condition},} \label{mstcline}
\State $\alpha_{k+1} := \theta \alpha_{k}$
\State $\delta_{k+1} \gets \tau_1 \norm{\br_{k+1}}$\label{resetline}
\Else
\State $\alpha_{k+1} := \alpha_{k}$
\EndIf 
\State $k \gets k+1$
\EndWhile
\State $\mathbf{u}_\varepsilon := \mathbf{u}_k$ 
\end{algorithmic}
\label{alg:iestsolve}
\end{algorithm}

To this end, as in the proof of Theorem \ref{rankest2}, it suffices to show that \eqref{mod_stsolve_condition} implies $\norm{\bu_{k+1} - \bu^{\alpha_k}} \leq \nu \norm{\bu_{k+1} - \bu^*}$. On the one hand, by the construction of $\br_k$ and the standard error estimate for fixed point iterations, we have
\begin{align*}
  \norm{\bu_{k+1} - \bu^{\alpha_k}} & \leq \hat\rho(\tau_2) \norm{\bu_k - \bu^{\alpha_k}} \leq \frac{\hat\rho(\tau_2)}{1 - \rho} \norm{\bS_{\alpha_k} (\bu_k - \mu \mathbf{\bar r}_k ) - \bu_k } \\ 
  & \leq \frac{\hat\rho(\tau_2) (1 + \tau_2)}{1-\rho} \norm{\bu_{k+1} - \bu_k} \,,
\end{align*}
and on the other hand, by the construction of $\br_{k+1}$, 
\begin{equation*}
  \norm{\bu_{k+1} - \bu^*} \geq \Gamma^{-1} \norm{\mathbf{\bar r}_{k+1}} \geq  (1- \tau_1) \Gamma^{-1} \norm{\br_{k+1}} \,.
\end{equation*}
Combining these two estimates, we find that \eqref{mod_stsolve_condition} implies \eqref{cond}. With this implication and the modified estimates \eqref{perturbedcontr} and \eqref{mod_contr}, one can now follow the proof of Theorem \ref{rankest2} to obtain the same statements.

There are two additional checks in the algorithm to ensure that the $\delta_k$ cannot become arbitrarily small. On the one hand, when the condition in line \ref{residualexitline} of Algorithm \ref{alg:iestsolve} is satisfied, then $\norm{\Acal \bu_{k+1} - \mathbf{f}} \leq \gamma \varepsilon$, which implies $\norm{\bu_{k+1}-\bu^*} \leq \varepsilon$, and we can therefore stop the iteration. 

On the other hand, if the loop in line \ref{secondetacondline} exits because the second condition with the constant
\begin{equation}\label{const2}
  D := \min \biggl\{  \frac{(1-\tau_1) \tau_2 B}{(1 + \tau_1 + \Gamma B )\mu}   \, ,    \;
    \frac{  \rho \nu \tau_2 (1-\tau_1)^2   }{  \bigl(  \rho(1+\tau_1)(1+\tau_2)  + \nu (1-\tau_1) (1-\rho) \bigr)  \mu }   \biggr\}  
\end{equation}
is violated, that is, if
\begin{equation}\label{viol2}
 \delta_k \leq D \norm{\br_k} \,,
\end{equation}
then the condition $\norm{ \bu_{k+1} - \bu^{\alpha_k} } \leq \nu \norm{\bu_{k+1} - \bu^*}$ as in \eqref{nueq} is satisfied and condition \eqref{mod_stsolve_condition} in line \ref{mstcline} is guaranteed to hold, which means that $\alpha_{k}$ will be decreased.
To see this, note first that
\begin{equation} \label{newreslower}
  \norm{\br_{k+1}} \geq (1 + \tau_1)^{-1} \bigl(  (1 - \tau_1) \norm{\br_k}  - \Gamma \norm{\bu_{k+1} - \bu_k }  \bigr) \,,  
\end{equation}
and since the first condition in line \ref{secondetacondline} still holds, we have $\norm{\bu_{k+1} - \bu_k} \leq \mu \tau_2^{-1} \delta_k$. Therefore \eqref{viol2} implies in particular
\[     \bigl(  1+ (1 + \tau_1)^{-1} \Gamma B \bigr)  \norm{\bu_{k+1} - \bu_k}  \leq  B ( 1 + \tau_1)^{-1} (1-\tau_1) \norm{\br_k} \,,     \]
which combined with \eqref{newreslower} implies \eqref{mod_stsolve_condition}. Furthermore, \eqref{viol2} also yields, with the second case in the minimum in \eqref{const2}, the estimate
\begin{equation}\label{viol2-2}    \frac{\rho}{1 - \rho} (\tau_2^{-1} \mu + \mu) \delta_k  \leq \nu \Gamma^{-1} ( 1 - \tau_1 ) ( 1 + \tau_1)^{-1} \bigl( (1 - \tau_1) \norm{\br_k} - \Gamma \mu  \tau_2^{-1} \delta_k \bigr) \,.   
\end{equation}
Since $\nu \norm{\bu_{k+1} - \bu^*} \geq \nu \Gamma^{-1} \norm{\mathbf{\bar{r}}_k } \geq \nu\Gamma^{-1} (1-\tau_1) \norm{\br_{k+1}}$, by \eqref{newreslower} the right hand side in \eqref{viol2-2} can be estimated from above by $\nu \norm{\bu_{k+1} - \bu^*}$. For the left hand side, we have
\begin{multline*}
  \frac{\rho}{1 - \rho} (\tau_2^{-1} \mu + \mu) \delta_k \geq  \frac{\rho}{1 - \rho} \bigl( \norm{\bu_{k+1} - \bu_k}  +  \mu \norm{\br_k - \mathbf{\bar r}_k}  \bigr)
     \\    \geq \frac{\rho}{1 - \rho}  \norm{\bS_{\alpha_k} ( \bu_k - \mu \mathbf{\bar r}_k) - \bu_k}  \geq \norm{\bu_{k+1} - \bu^{\alpha_k}}  \,,
\end{multline*}
and from \eqref{viol2-2} altogether we obtain $\norm{ \bu_{k+1} - \bu^{\alpha_k} } \leq \nu \norm{\bu_{k+1} - \bu^*}$ as required.

Note that as a consequence of this construction, the $\delta_k$ obtained in Algorithm \ref{alg:iestsolve} remain proportional to $\norm{\br_k}$ during the iteration.

\section{Numerical Experiments}\label{sec:numexp}

In principle, Algorithms \ref{alg:stsolve} and \ref{alg:iestsolve} can be applied to quite general discretized elliptic problems, since only bounds on the spectrum of the discrete operator $\Acal$ are required. For our numerical tests, we choose a particular setting where we have a suitable method for preconditioning with explicit control of the resulting condition numbers available: we test Algorithm \ref{alg:iestsolve} on a discretized Poisson problem with homogeneous Dirichlet boundary conditions
\[   -\Delta u = f \quad \text{on $(0,1)^d$,}  \]
using similar techniques as for the adaptive treatment in \cite{BD2}, with the difference that we now use a wavelet Galerkin discretization with the basis functions chosen in advance. 

We shall now briefly describe how the discrete operator $\Acal$ is obtained as a symmetrically preconditioned Galerkin discretization in a tensor product wavelet basis. Starting from an orthonormal basis of sufficiently regular (multi-)wavelets $\{ \psi_\nu \}_{\nu \in\nabla}$ of $L^2(0,1)$, we obtain a tensor product orthonormal basis $\{ \Psi_\nu \}_{\nu \in \nabla^d}$ of $L^2((0,1)^d)$ with $\Psi_\nu := \psi_{\nu_1}\otimes \cdots\otimes \psi_{\nu_d}$, such that the rescaled basis functions $\omega_\nu^{-1} \Psi_\nu$, $\nu\in\nabla^d$, where
\[    \omega_\nu := \Bigl(  \norm{\psi_{\nu_1}}^2_{H^1_0(0,1)}  + \ldots +  \norm{\psi_{\nu_d}}^2_{H^1_0(0,1)} \Bigr)^{\frac12} \,,   \]
form a Riesz basis of $H^1_0((0,1)^d)$. We now pick a fixed finite subset $\Lambda_1 \subset \nabla$ and set $\Lambda := \Lambda_1 \times \cdots\times \Lambda_1 \subset \nabla^d$. Furthermore, we use the family of low-rank approximate diagonal scaling operators $\mathbf{\tilde S}_n^{-1}$, $n\in\N$, constructed in \cite{BD2}: we choose a $\bar{\delta}\in(0,1)$ and then take $\bar{n}$ according to \cite[Theorem 4.1]{BD2} such that 
\[   (1- \bar{\delta}) \bignorm{\diag(\omega_\nu^{-1}) \bv } \leq \bignorm{\mathbf{\tilde S}^{-1}_{\bar{n}} \bv} \leq (1+\bar{\delta}) \bignorm{\diag(\omega_\nu^{-1}) \bv }  \]
for all sequences $\bv$ supported on $\Lambda$. With
\begin{equation*}
   \mathbf{\hat T}_\Lambda := \Bigl(  \sum_{i=1}^d \langle \partial_i \Psi_\nu,\partial_i \Psi_\mu\rangle_{L^2}  \Bigr)_{\lambda,\nu\in\Lambda}  \,,\quad 
     \mathbf{\hat f}_\Lambda := \bigl(  \langle f, \Psi_\nu\rangle  \bigr)_{\nu\in\Lambda}\,,
\end{equation*}
we then set
\begin{equation*}
  \Acal :=  \mathbf{\tilde S}^{-1}_{\bar{n}} \mathbf{\hat T}_\Lambda \mathbf{\tilde S}^{-1}_{\bar{n}} \,,\quad \mathbf{f} := \mathbf{\tilde S}^{-1}_{\bar{n}} \mathbf{\hat f}_\Lambda \,.
\end{equation*}
Thus $\bu^* = \mathbf{\tilde S}_{\bar{n}} \mathbf{\hat T}_\Lambda^{-1}\mathbf{\hat f}_\Lambda$, where the additional scaling by $\mathbf{\tilde S}_{\bar{n}}$ yields convergence of the scheme in $H^1$-norm at a controlled rate. An approximation of $\mathbf{\hat T}_\Lambda^{-1}\mathbf{\hat f}_\Lambda$, which in turn is a Galerkin approximation of the sequence of $L^2$-coefficients $\langle u,\Psi_\nu\rangle$ of the true solution, can then be recovered by applying $\mathbf{\tilde S}^{-1}_{\bar{n}}$ to the computed $\bu_\varepsilon$. For $\Acal$, one can obtain accurate bounds for $\gamma$, $\Gamma$, and in particular,
\[   \kappa \leq \frac{(1+\bar{\delta})^2}{(1-\bar{\delta})^2}  \operatorname{cond}_2 \bigl(  \omega_\lambda^{-1} \langle \psi_\lambda', \psi_\nu'\rangle \omega_\nu^{-1} \bigr)_{\lambda,\nu\in\Lambda_1}  \,.  \]
In our numerical tests, as in \cite{BD2} we take $f = 1$ and use the piecewise polynomial, continuously differentiable orthonormal multiwavelets of order 7 constructed in \cite{DGH:99}. The univariate index set $\Lambda_1$ comprises all multiwavelet basis functions on levels $0,\ldots, 4$, which yields $\#(\Lambda_1) = 224$. The unspecified constants in Algorithm \ref{alg:iestsolve} are chosen as $\theta := \frac34$, $\omega := \frac12$, $\nu := \frac9{10}$, $\alpha_0 := \frac12 \mu \norm{\mathbf{f}}$, and we take $\bar{\delta} := \frac1{10}$.

Note that since the resulting diagonal scalings $\mathbf{\tilde S}^{-1}_{\bar{n}}$ consist of 10 separable terms, a naive direct application of $\Acal$ could increase the hierarchical ranks of a given input $\bv$ by a factor of up to 200; the observed ranks required for accurately approximating $\Acal\bv$, however, are much lower.
Therefore we use the recompression strategy described in \cite[Section 7.2]{BD2} for an approximate evaluation of $\Acal \bv$ with prescribed tolerance in order to avoid unnecessarily large ranks in intermediate quantities. In this setting, the inexact residual evaluation in Algorithm \ref{alg:iestsolve} is thus of crucial practical importance.

We compare the computed solutions to a very accurate reference solution of the discrete problem obtained by an exponential sum approximation $\mathbf{\hat u}_0 \approx \mathbf{\hat T}_\Lambda^{-1}\mathbf{\hat f}_\Lambda$, see \cite{Grasedyck:04,Hackbusch:05}. The error to the reference solution is computed as $\text{err}_k := \norm{\diag(\omega_\nu) (\mathbf{\tilde S}^{-1}_{\bar{n}} \bu_k - \mathbf{\hat u}_0)}$, which is proportional to the error in $H^1$-norm of the corresponding represented functions. The quantity $\text{err}_k$ thus serves as a substitute for the difference in the relevant norm of $\bu_k$ to the exact solution of the discretized problem.

\begin{figure}[ht]
\begin{tabular}{ccc}
\hspace{-20pt}  \includegraphics[width=5.4cm]{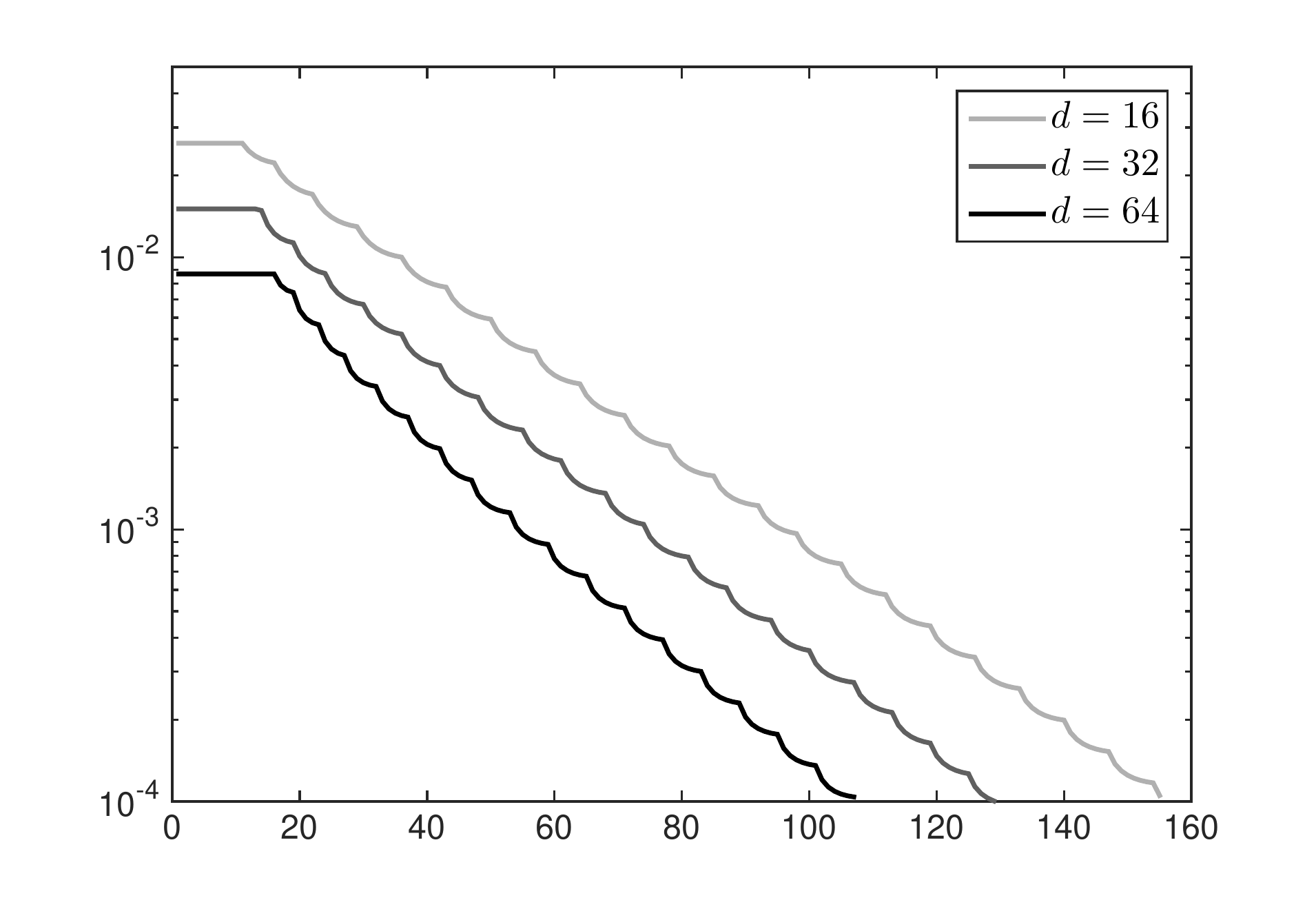}  &  \hspace{-24pt}\includegraphics[width=5.4cm]{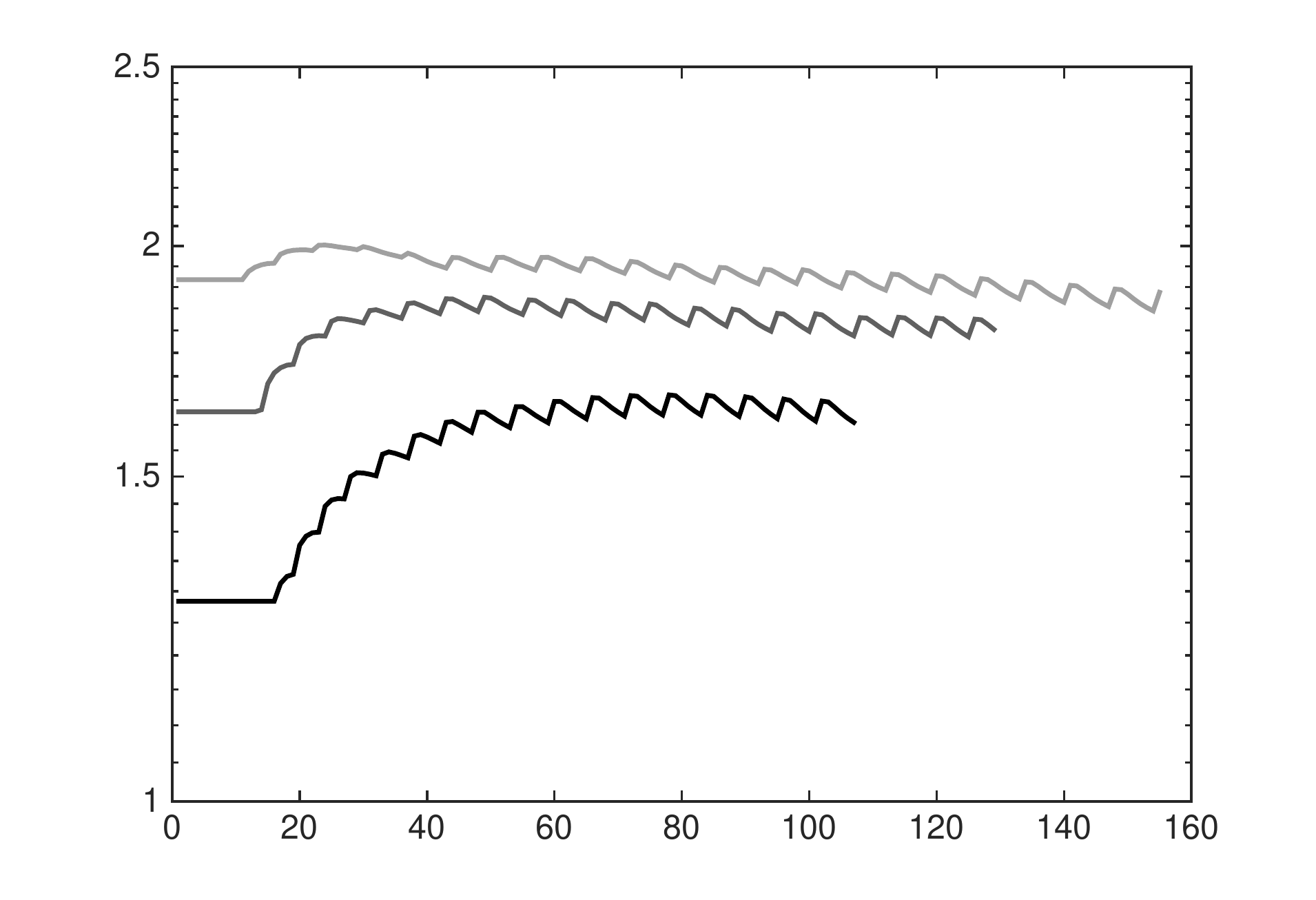} & \hspace{-20pt}\includegraphics[width=5.4cm]{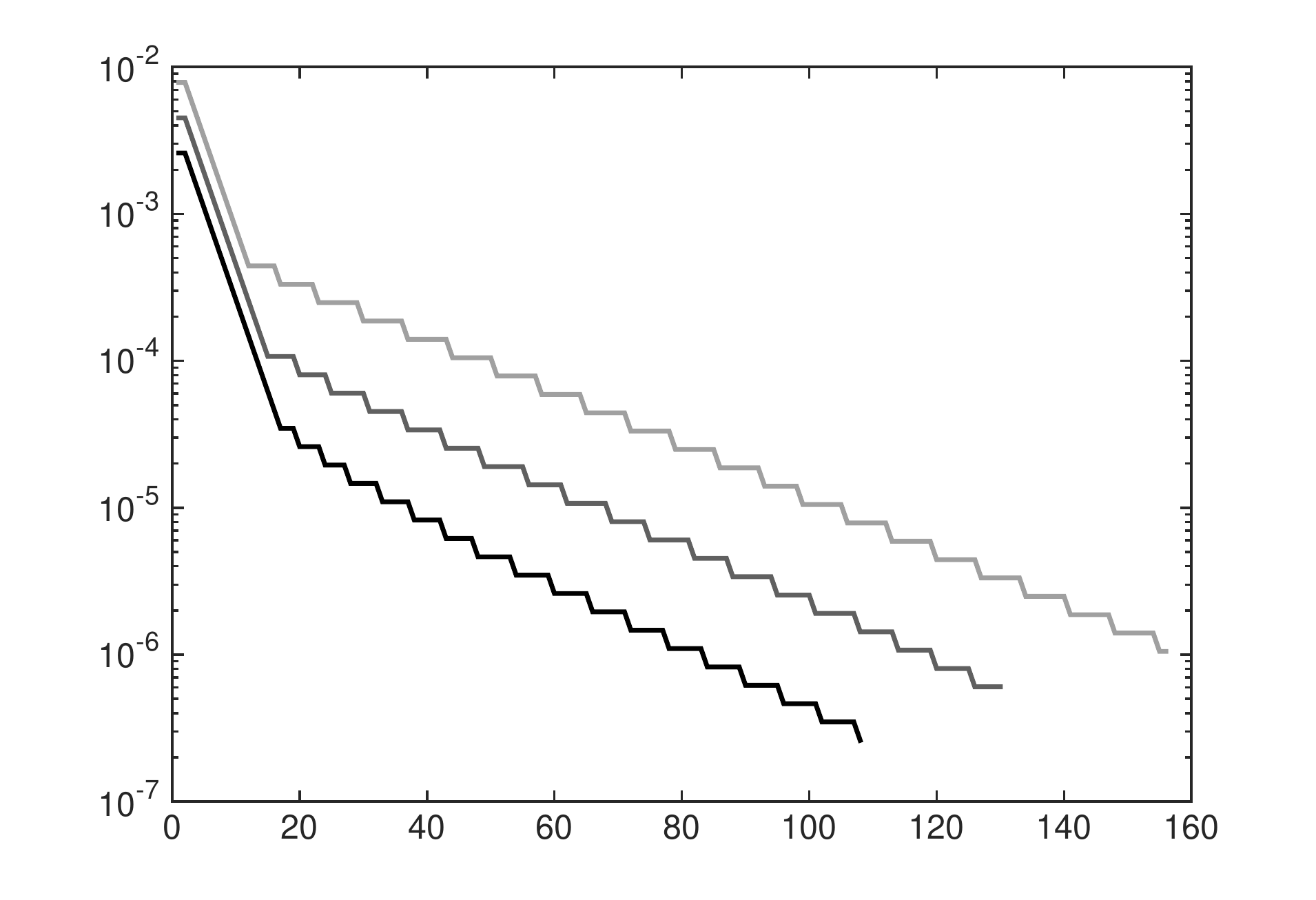} \\[-6pt]
    $\norm{\br_k}$   &  $\text{err}_k / \norm{\br_k} $  & $\alpha_k$
\end{tabular}\vspace{-3pt}
\caption{From left to right (versus iteration number $k$): computed discrete residual norms $\norm{\br_k}$, ratios of differences to reference solution $\text{err}_k$ to $\norm{\br_k}$, corresponding thresholding parameters $\alpha_k$; each for $d=16$ (light grey), $d=32$ (dark grey), $d=64$ (black).}\label{fig:res}
\end{figure}
\begin{figure}[ht]
\begin{tabular}{ccc}
\hspace{-22pt}  \includegraphics[width=5.4cm]{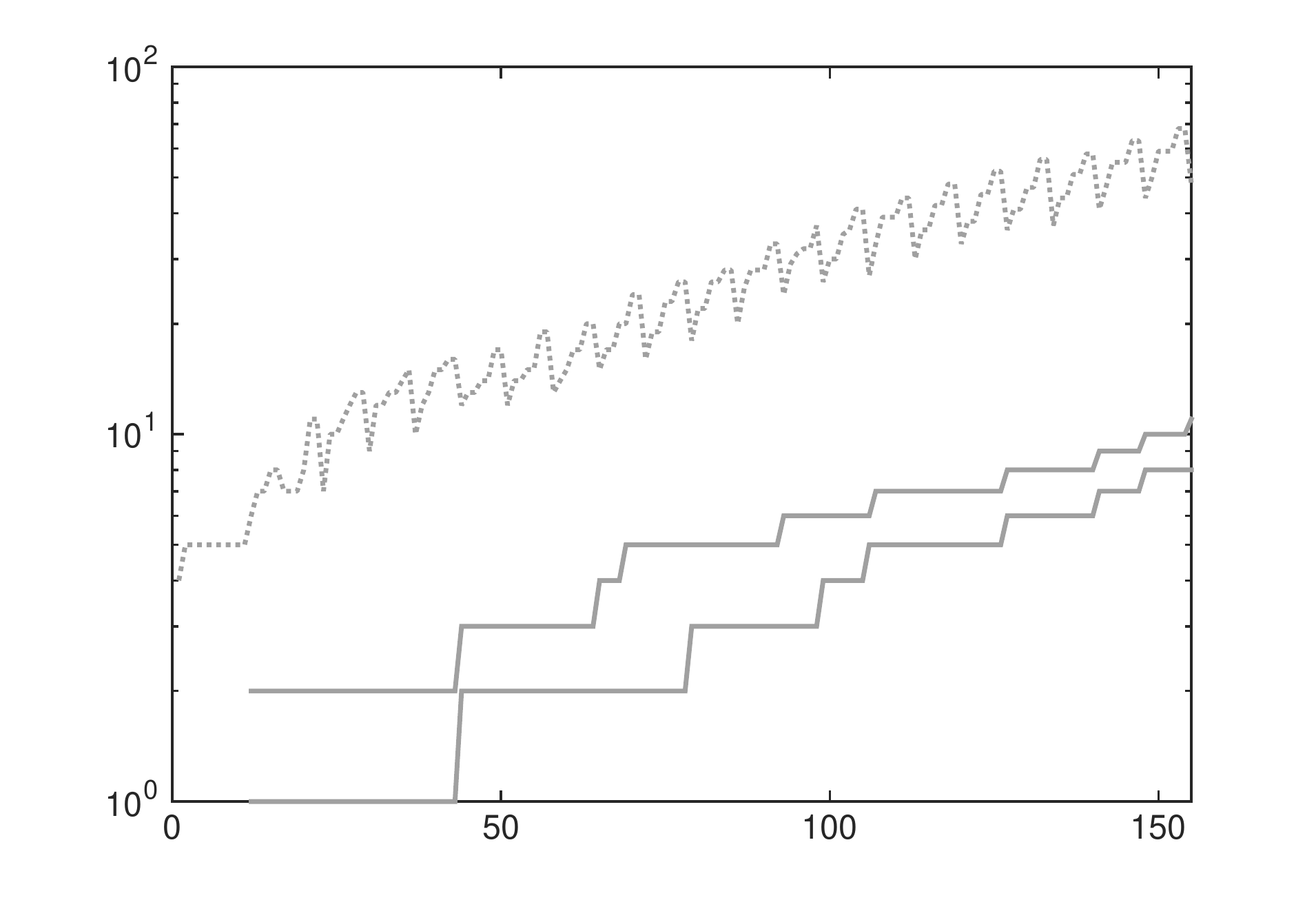}  & \hspace{-24pt} \includegraphics[width=5.4cm]{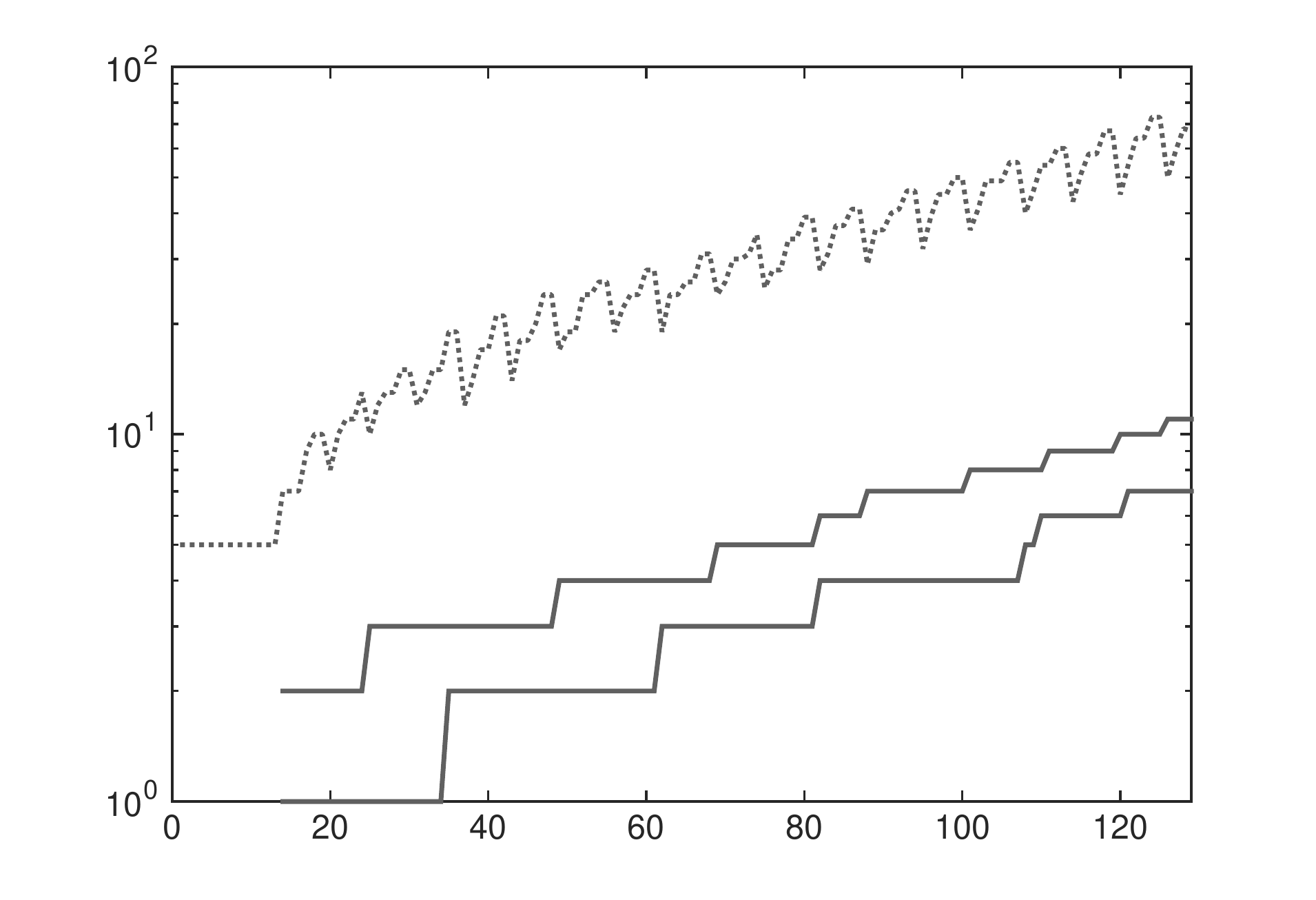} & \hspace{-26pt} \includegraphics[width=5.4cm]{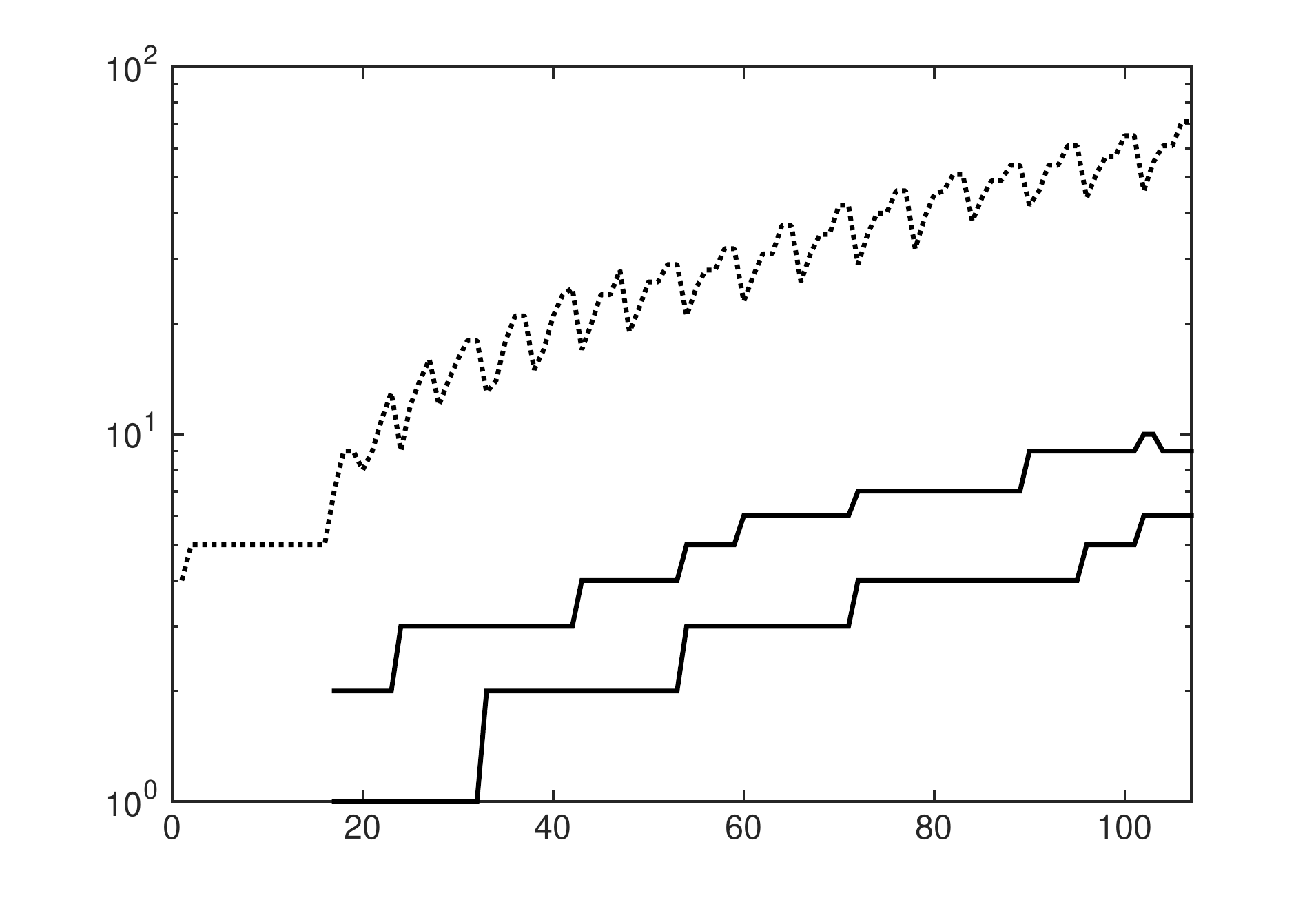} \\[-6pt]
    $d=16$   &  $d=32$  & $d=64$
\end{tabular}\vspace{-6pt}
\caption{Solid lines: maximum and minimum hierarchical ranks of iterates, i.e., $\min_t \rank_t(\bu_k)$ and $\max_t \rank_t(\bu_k)$; dotted lines: maximum ranks of computed residuals, $\max_t \rank_t(\br_k)$; each versus $k$.}\label{fig:ranks}
\end{figure}

The numerical results for $d=16$, $d=32$, and $d=64$ are shown in Figures \ref{fig:res} and \ref{fig:ranks}. It can be observed in Figure \ref{fig:res} that the norm of the solution of the problem decreases slightly with increasing $d$; apart from this, the iteration behaves very similarly for the different values of $d$, producing in particular a monotonic decrease of discrete residual norms. As expected, these values also remain uniformly proportional, up to very moderate constants, to the $H^1$-difference to the reference solution. The values $\alpha_k$ can be seen to first decrease in every step as long as $\bu_k = 0$; subsequently, they decrease in a regular manner after an essentially constant number of iterations. As one would also expect, the final value of $\alpha_k$ needs to be slightly smaller for larger $d$.

Figure \ref{fig:ranks} shows the maximum and minimum hierarchical ranks of the computed iterates (whose difference grows slightly with increasing $d$) compared to the ranks of the corresponding computed discrete residuals $\br_k$, clearly demonstrating the reduced rank increase relative to $\bu_k$ that we obtain by the approximate residual evaluation. The additional variation in the residual ranks is a consequence of the fact that the the differences $\norm{\bu_{k+1} - \bu_k}$ decrease as long as $\alpha_k$ remains constant, enforcing a more accurate residual evaluation. As soon as the thresholding parameter changes, the accuracy requirement is subsequently relaxed again by line \ref{resetline} in Algorithm \ref{alg:iestsolve}, since the values $\norm{\bu_{k+1} - \bu_k}$ are again increased when $\bu^{\alpha_k}$ changes. Note furthermore that the ranks show little variation with increasing $d$, which is substantially more favorable than the quadratic increase with $d$ that is possible in the estimates \eqref{thmwlp} and \eqref{thmexp} of Theorem \ref{rankest2}.

\section{Conclusion}

We have constructed an iterative scheme for solving linear elliptic operator equations in hierarchical tensor representations. This method guarantees linear convergence to the solution $\bu^*$ as well as quasi-optimality of the tensor ranks of all iterates, and is universal in the sense that no a priori knowledge on the tensor approximability of $\bu^*$ is required.

However, if it is known that the hierarchical singular values of $\bu^*$ have, for instance, exponential-type decay, then Theorem \ref{rankest1} shows that one can obtain the same properties by a priori prescribing $\alpha_k$ that decrease geometrically at some rate $\tilde\rho > \rho$. In such a setting, this simpler approach with a priori choice may thus be a viable alternative.

Since the given a priori choices of thresholding parameters work for quite general contractive fixed point mappings, the construction of schemes that make this choice a posteriori may be possible for more general cases than the linear elliptic one treated here. In this regard, note that although we have always assumed for ease of presentation that the considered operator $\Acal$ is also symmetric, this is not essential.

In this work, we have considered fixed discretized problems, but we expect that the basic strategy proposed here can also be used in the context of adaptive discretizations.
Moreover, there may exist other related soft thresholding procedures for tensors than the sequential approach underlying our construction that retain the required features.

\bibliography{st}

\end{document}